\pdfoutput=1
\RequirePackage{ifpdf}
\ifpdf 
\documentclass[pdftex]{sigma}
\else
\documentclass{sigma}
\fi

\numberwithin{equation}{section}

\newtheorem{Theorem}{Theorem}[section]
\newtheorem{Corollary}[Theorem]{Corollary}
\newtheorem{Lemma}[Theorem]{Lemma}
\newtheorem{Proposition}[Theorem]{Proposition}
\newtheorem{Conjecture}[Theorem]{Conjecture}
 { \theoremstyle{definition}
\newtheorem{Definition}[Theorem]{Definition}
\newtheorem{Example}[Theorem]{Example}
\newtheorem{Remark}[Theorem]{Remark} }

\newcommand{\Z}{\mathbb{Z}}

\newcommand{\g}{\mathfrak{g}}
\newcommand{\h}{\mathfrak{h}}

\begin{document}

\allowdisplaybreaks

\newcommand{\arXivNumber}{1909.13588}

\renewcommand{\thefootnote}{}

\renewcommand{\PaperNumber}{014}

\FirstPageHeading

\ShortArticleName{Short Star-Products for Filtered Quantizations, I}

\ArticleName{Short Star-Products for Filtered Quantizations, I\footnote{This paper is a~contribution to the Special Issue on Algebra, Topology, and Dynamics in Interaction in honor of Dmitry Fuchs. The full collection is available at \href{https://www.emis.de/journals/SIGMA/Fuchs.html}{https://www.emis.de/journals/SIGMA/Fuchs.html}}}

\Author{Pavel ETINGOF and Douglas STRYKER}
\AuthorNameForHeading{P.~Etingof and D.~Stryker}
\Address{Department of Mathematics, MIT, Cambridge, MA 02139, USA}
\Email{\href{mailto:etingof@math.mit.edu}{etingof@math.mit.edu}, \href{mailto:stryker@mit.edu}{stryker@mit.edu}}

\ArticleDates{Received October 01, 2019, in final form March 01, 2020; Published online March 11, 2020}

\Abstract{We develop a theory of short star-products for filtered quantizations of graded Poisson algebras, introduced in 2016 by Beem, Peelaers and Rastelli for algebras of regular functions on hyperK\"ahler cones in the context of 3-dimensional $N=4$ superconformal field theories [Beem C., Peelaers W., Rastelli L., \textit{Comm. Math. Phys.} \textbf{354} (2017), 345--392]. This appears to be a new structure in representation theory, which is an algebraic incarnation of the non-holomorphic ${\rm SU}(2)$-symmetry of such cones. Using the technique of twisted traces on quantizations (an idea due to Kontsevich), we prove the conjecture by Beem, Peelaers and Rastelli that short star-products depend on finitely many parameters (under a natural nondegeneracy condition), and also construct these star products in a number of examples, confirming another conjecture by Beem, Peelaers and Rastelli.}


\Classification{06B15; 53D55}

\rightline{\it To Dmitry Borisovich Fuchs on his 80th birthday with admiration}

\renewcommand{\thefootnote}{\arabic{footnote}}
\setcounter{footnote}{0}

\section{Introduction}

The goal of this paper is to develop a theory of short star-products for filtered quantizations of graded Poisson algebras, introduced in 2016 by Beem, Peelaers and Rastelli for algebras of regular functions on hyperK\"ahler cones in the context of 3-dimensional $N=4$ superconformal field theories~\cite{BPR}.

Namely, let $A$ be a commutative ${\mathbb Z}_{\ge 0}$-graded ${\mathbb C}$-algebra with a Poisson bracket $\lbrace\,,\,\rbrace$ of deg\-ree~$-2$, and let~$*$ be a star-product on $A$ quantizing $\lbrace\,,\,\rbrace$. This means that~$*$ is an associative product on~$A$ defined by
\begin{gather*}
a*b=ab+C_1(a,b)+C_2(a,b)+\cdots,\qquad a,b\in A,
\end{gather*}
where $C_k\colon A\otimes A\to A$ are bilinear maps of degree $-2k$ such that $C_1(a,b)-C_1(b,a)=\lbrace a,b\rbrace$.\footnote{Thus, we set the Planck constant $\hbar$ to be~$1$, which can be done without loss of generality because of the grading on $A$.} For degree reasons, for any homogeneous $a,b\in A$, one has $C_k(a,b)=0$ for $k>\frac{\deg(a)+\deg(b)}{2}$, so the sum makes sense.

Star-products arise naturally from filtered quantizations of $A$. Namely, let $d\colon A\to A$ be the degree operator, i.e., $d(a)=ma$ for $a\in A_m$, and let $s=(-1)^d\colon A\to A$ be the operator such that $s(a)=(-1)^ma$ for $a\in A_m$.\footnote{As customary in physical literature, we will abuse notation by writing $d$ both for the degree of an element of~$A$ and the degree operator $A\to A$ whose eigenvalues are such degrees.} Then $s$ is a Poisson automorphism of~$A$ defining an action of ${\mathbb Z}/2$ on $A$. Let ${{\mathbf A}}$ be a ${\mathbb Z}/2$-equivariant filtered quantization of $A$, and let $\phi\colon A\to {{\mathbf A}}$ be a ${\mathbb Z}/2$-equivariant quantization map, i.e., a ${\mathbb Z}/2$-equivariant filtration-preserving linear map such that $\operatorname{gr}\phi={\rm id}$ (it attaches to a ``classical observable'' $a\in A$ the corresponding ``quantum observable'' $\widehat{a}:=\phi(a)$). Then the formula
\begin{gather*}
a*b:=\phi^{-1}(\phi(a)\phi(b))
\end{gather*}
defines a star-product on $A$.

Conversely, any star-product on $A$ defines a filtered quantization of $A$ realized as a new product on the same vector space, and this is a traditional method of quantizing Poisson algebras. For example, the well known Moyal--Weyl star-product (Example~\ref{MV} below) gives a quantization of the standard $2n$-dimensional classical phase space which goes back to the early years of quantum mechanics.

This classical story has recently had an exciting new development. Namely, motivated by 3-dimensional N=4 superconformal field theories, in 2016 Beem, Peelaers and Rastelli~\cite{BPR} considered a special class of star-products which we call {\it short}, i.e., such that for any homogeneous $a,b\in A$ one has $C_k(a,b)=0$ for any $k>\min(\deg(a),\deg(b))$. In other words, the degree of any term in $a*b$ is at least $|\deg(a)-\deg(b)|$ (in~\cite{BPR} this is called {\it the truncation condition}). Beem, Peelaers and Rastelli showed how short star-products appear naturally in 3-dimensional N=4 superconformal field theories, defining filtered quantizations of the corresponding Higgs and Coulomb branches $X$, and conjectured that any {\it even} filtered quantization of $A={\mathcal O}(X)$ (i.e., a ${\mathbb Z}/2$-equivariant quantization ${{\mathbf A}}$ equipped with a filtered antiautomorphism $\sigma\colon {{\mathbf A}}\to {{\mathbf A}}$ such that $\operatorname{gr}(\sigma)={\rm i}^d$ and $\sigma^2=s$) admits an {\it even} short star-product (i.e., such that $C_k(a,b)=(-1)^kC_k(b,a)$). Moreover, they predicted that such star-products depend on finitely many parameters. Finally, they conjectured the existence of an even short star-product satisfying a positivity condition; namely, this condition should hold for the star-product arising from the physical theory.

There are several examples (given in \cite{BPR}) in which this conjecture is easy to prove, as the even short star-product is uniquely determined by the symmetry of the situation. For instance, if $V$ is a finite dimensional symplectic vector space and $A=\mathcal{O}(V)$ then the symmetric Moyal--Weyl star-product on $A$ is short and even. Since this star-product is ${\rm Sp}(V)$-invariant, for any finite subgroup $G\subset {\rm Sp}(V)$ this gives an even short star-product on ${\mathcal O}(V/G)=\mathcal{O}(V)^G$. Also, if $X\subset \mathfrak g^*$ is the closure of the minimal nilpotent orbit of a simple Lie algebra $\mathfrak g$, then any filtered (even) quantization of $\mathcal{O}(X)$ (with degrees multiplied by $2$) may be defined by a $\g$-invariant (even) star-product (see \cite{L1}), which is short for representation-theoretic reasons.\footnote{The quantization of the minimal orbit is the quotient of ${\rm U}(\g)$ by the so-called Joseph ideal. This quantization (and ideal) is unique if $\g\ne \mathfrak{sl}_n$ and depends on one parameter $\lambda$ for $\g= \mathfrak{sl}_n$. The even quantization is unique except $\g=\mathfrak{sl}_2$, where all quantizations are even. E.g., for $\g=\mathfrak{sl}_n$ with $n\ge 3$, the quantization is even if and only if $\lambda=0$.}

In particular, this applies to the $A_1$-singularity, which corresponds to the case of $\mathfrak g=\mathfrak{sl}_2$. Namely, in this case $X$ is the usual quadratic cone in~${\mathbb C}^3$, so $\mathcal O(X)=\oplus_{m\in 2{\mathbb Z}_{\ge 0}}V_m$, where $V_m$ is the irreducible representation of ${\rm SU}(2)$ with highest weight $m$, and the shortness property comes from the fact that if~$V_k$ is contained
in $V_m\otimes V_n$ then $k\ge |m-n|$ (the Clebsch--Gordan rule). In fact, as explained in~\cite{BPR}, the same mechanism, but for the non-holomorphic ${\rm SU}(2)$-symmetry of the hyperK\"ahler cone acting by rotations on its 2-sphere of complex structures, is responsible for appearance of short star-products in 3-dimensional $N=4$ superconformal field theories.

However, the shortness condition for a star-product is a very strong restriction, and from the purely mathematical viewpoint the existence of short star-products for a given quantization is far from obvious. In general it is not even clear if there exists any short star-product at all. In fact, the existence of a~short star-product is already non-obvious for quantizations of Kleinian singularities beyond the~$A_1$ case, although some low-degree calculations of~\cite{BPR} for~$A_2$ and~$A_3$ showed that the existence of positive even short star-products (for even quantizations) is highly plausible, and this was confirmed later in \cite{DFPY1, DPY} for~$A_n$ for all~$n$. Also, examples of short star-products for $D_n$ singularities are constructed in \cite[Section~5.2.1]{DFPY2}.

The goal of this paper is to partially prove the conjectures from~\cite{BPR}.
Namely, let us say that a short star-product $*$ is {\it nondegenerate}
if the bilinear form $(a,b)={\rm CT}(a*b)$ on $A$ (where ${\rm CT}$ stands for the constant term)
 is nondegenerate. This is a natural condition from the physics point of view; e.g.,
it follows from the positivity condition, saying that the Hermitian version of this form is positive definite. Our first main result is that nondegenerate short star-products for hyperK\"ahler cones depend on finitely many parameters, confirming a conjecture from~\cite{BPR}.
The proof is based on the idea of Kontsevich to express nondegenerate short star-products in terms of {\it nondegenerate twisted traces}, which can then be understood as elements of a certain zeroth Hochschild homology group, as well as a number of results on symplectic singularities~\cite{Ka,L2,Na3}, hyperK\"ahler manifolds~\cite{HKLR} and Poisson homology~\cite{ES}.

Our second main result is the existence and classification of nondegenerate short star-products in a number of examples of hyperK\"ahler cones, such as quotient singularities (where the quantizations are spherical symplectic reflection algebras) and nilpotent cones of simple Lie algebras (where the quantizations are quotients of the enveloping algebra by a central character). We also explain how short star-products on ${\mathcal O}(X)$ arise from representation theory of its quantizations.

To summarize, short star-products appear to be a new structure in representation theory, which is, in a sense, an algebraic incarnation of the non-holomorphic ${\rm SU}(2)$-symmetry of a~hyperK\"ahler cone.
While much of representation theory is the study of quantizations of hyperK\"ahler cones, this symmetry has not been studied much by representation theorists, perhaps because it is not algebraic (as in the context of representation theory, hyperK\"ahler cones are usually treated just as algebraic symplectic varieties). However, the theory of short star-products suggests that the full hyperK\"ahler structure and the associated non-holomorphic ${\rm SU}(2)$ sym\-met\-ry are also quite relevant and worth taking into account in representation theory, as pointed out in \cite[footnote on p.~3]{BPR}. The goal of this paper is to provide the tools to do so, by connecting short star-products with more traditional and well understood objects in representation theory (such as characters of representations and Hochschild homology).

The organization of the paper is as follows. Section~\ref{section2} contains preliminaries on short star-products, filtered quantizations, quantization maps, the evenness condition, and also some basics about symplectic singularities and hyperK\"ahler cones. In Section~\ref{section3} we introduce nondegenerate twisted traces, connect them to nondegenerate short star-products, and show that they depend on a finite number of parameters, establishing our first main result. We also discuss Hermitian and quaternionic structures on short star-product quantizations, which are needed to study their positivity properties. Finally, in Section~\ref{section4} we consider a number of examples of hyperK\"ahler cones and construct and classify nondegenerate short star-products for them, establishing our second main result. In particular, we establish rationality of reduced characters of Verma modules (with arbitrary insertions) in category $\mathcal{O}$ defined in \cite{BLPW,L4}.

In the second part of this work, joint with Eric Rains, we plan to extend the existence and classification of nondegenerate short star-products to a larger class of hyperK\"ahler cones (including Nakajima quiver varieties) and discuss the relation with representation theory in more depth. We also plan to give some explicit computations of short star-products for type A singularities, connecting this work with the computational results of \cite{BPR,DPY}, and study the orthogonal polynomials that arise from them. Finally, we plan to study the positivity properties of short star-products for type A singularities and in general, fulfilling the program of~\cite{BPR}.

{\bf Dedication} ({\sl from Pavel Etingof}). It is my great pleasure to dedicate this paper to my teacher Dmitry Borisovich Fuchs on his 80th birthday. I have fond memories of his topology lessons from over 30 years ago. Dmitry Borisovich taught me to appreciate the beauty of interaction between different fields of mathematics, which I enjoyed to a full extent while working on this paper.

\section{Preliminaries}\label{section2}

\subsection{Star-product quantization of a graded Poisson algebra}
Let $A$ be a commutative ${\mathbb C}$-algebra with a $\Z_{\geq 0}$ grading given by
\[
A = \bigoplus_{d \geq 0} A_d.
\]
Let $\{~,~\}$ be a Poisson bracket of degree $-2$ on $A$; namely,
\[
\{~,~\} \colon \ A_n \otimes A_m \to A_{n+m-2}.
\]
A basic example is the symmetric algebra $A=SV$ of a finite dimensional symplectic vector space $V$
with the Poisson bracket defined by the symplectic form.

Note that this setting includes the case of graded Poisson algebras with Poisson bracket of degree $-1$ (such as the symmetric algebra of a Lie algebra, the algebra of regular functions on the nilpotent cone, etc.), by multiplying the degrees by $2$.\footnote{In fact, this theory can be developed for Poisson brackets of any negative degree, but for most applications it is sufficient to consider degree $-2$, so we restrict our attention to this case.}

\begin{Definition}
A \emph{star-product} on $A$ quantizing $\lbrace\,,\,\rbrace$ is an associative multiplication operation
\[
*\colon \ A \otimes A \to A
\]
such that for $a \in A_n$ and $b \in A_m$,
\[
a*b = \sum_{k = 0}^{\lfloor \frac{n+m}{2}\rfloor} C_k(a,b),
\]
where $C_k \colon A_n \otimes A_m \to A_{n+m-2k}$ are bilinear maps such that $C_0(a,b) = a b$ and
\begin{gather*}
C_1(a,b) - C_1(b,a) = \{a,b\}.
\end{gather*}
\end{Definition}

\begin{Definition} A star-product $*$ on $A$ is \emph{short} if
for any $m,n\in {\mathbb Z}_{\ge 0}$ and any $a \in A_n$, $b \in A_m$ one has $C_k(a,b) = 0$ for all $k > \min(n,m)$. In other words, $a * b$ has no terms in $A_d$ for $d < |n - m|$.
\end{Definition}

From now on assume that $A_0={\mathbb C}$ and $\dim A_i<\infty$ for all $i$. We define the following (in general, non-symmetric) inner product on $A$ induced by the star-product. Let $a \in A_n$ and $b \in A_m$. Then
\[
\langle a, b \rangle := \text{CT}(a * b),
\]
where $\text{CT}\colon A \to A_0$ is the constant term map, taking the term of degree zero. Note that if $*$ is short, then $\langle a, b \rangle = 0$ if $n \neq m$, i.e., the decomposition $A=\oplus_{d\ge 0}A_d$ is orthogonal.

\begin{Definition} A short star-product $*$ on $A$ is said to be {\it nondegenerate} if the corresponding inner product $\langle \,,\,\rangle$ is nondegenerate in each degree~$i$.
\end{Definition}

\subsection{Construction of star-products from a quantum algebra}\label{constru} Let $A$ be a ${\mathbb Z}_{\ge 0}$-graded algebra. Fix a filtered associative algebra
\[
{{\mathbf A}} = \bigcup_{d \geq 0} F_d{{\mathbf A}}
\]
such that its associated graded algebra is identified with the graded algebra $A$; namely,
\[
\operatorname{gr}{{\mathbf A}} = \bigoplus_{d \geq 0} F_{d}{{\mathbf A}}/F_{d-1}{{\mathbf A}}
\]
and $F_{d}{{\mathbf A}}/F_{d-1}{{\mathbf A}} \cong A_d$ for $d \geq 0$ (compatibly with multiplication) with $F_{-1}{{\mathbf A}} := 0$.
Assume that we have a filtration preserving ${\mathbb Z}/2$-action $s\colon {{\mathbf A}}\to {{\mathbf A}}$ such that the corresponding associated graded map $s\colon A\to A$ is given by $s=(-1)^d$, where $d\colon A\to A$ is the degree operator. In this case, it is easy to see that $[F_n{{\mathbf A}},F_m{{\mathbf A}}]\subset F_{n+m-2}{{\mathbf A}}$, so we have the leading coefficient map $\lbrace\,,\,\rbrace\colon A_n\otimes A_m\to A_{n+m-2}$. It is easy to check that $\lbrace\,,\,\rbrace$ is a Poisson bracket on $A$, and ${{\mathbf A}}$ is a~${\mathbb Z}/2$-equivariant filtered quantization of $\lbrace\,,\,\rbrace$.\footnote{Since all quantizations we will consider will be ${\mathbb Z}/2$-equivariant, we will not mention it explicitly from now on.}

\begin{Definition}A \emph{quantization map} is a ${\mathbb Z}/2$-equivariant isomorphism of vector spaces
\[
\phi \colon \ A \to {{\mathbf A}}
\]
such that $\phi(A_d) \subseteq F_d{{\mathbf A}}$ for $d\ge 0$ and $\operatorname{gr} \phi \colon A \to A$ is the identity map, where
\[
\operatorname{gr} \phi = \bigoplus_{d\geq 0} \operatorname{gr} \phi_d, \qquad \operatorname{gr} \phi_d \colon \ A_d \to F_d{{\mathbf A}}/F_{d-1}{{\mathbf A}}=A_d.
\]
\end{Definition}
Note that $\phi$ need not (and usually cannot) be a homomorphism of algebras, since the algebra~$A$ is commutative but ${{\mathbf A}}$ is not in general.

Given a quantization map $\phi\colon A \to {{\mathbf A}}$, we can define the star-product on $A$ by
\[
a * b := \phi^{-1}(\phi(a)\phi(b)).
\]
Conversely, given a star-product on $A$, we can set ${{\mathbf A}}:=(A,*)$ with the filtration induced by the grading, and $\phi:={\rm Id}$.

\begin{Proposition} This defines a pair of mutually inverse bijections between star-products on~$A$ and ${\mathbb Z}/2$-equivariant quantizations of $A$ equipped with a quantization map.
\end{Proposition}

\begin{proof} Straightforward.
\end{proof}

\subsection{Even star-products}

\begin{Definition} A star-product $*$ on $A$ is said to be {\it even} if $C_k(a,b)=(-1)^k C_k(b,a)$ for all $k\ge 0$.
\end{Definition}

In particular, for even star-products $C_1$ is skew-symmetric, so since $C_1(a,b)-C_1(b,a)=\lbrace a,b\rbrace$, we have $C_1(a,b)=\frac{1}{2}\lbrace a,b\rbrace$.

Let $*$ be an even star-product on $A$ and ${{\mathbf A}}=(A,*)$ be the corresponding quantum algebra.
Then the map $\sigma={\rm i}^d$ defines an antiautomorphism ${{\mathbf A}}\to {{\mathbf A}}$ such that $\sigma^2=s=(-1)^d$. A ${\mathbb Z}/2$-invariant quantization ${\mathbf A}$ of $A$ equipped with a~filtration-preserving antiautomorphism $\sigma$ such that $\sigma^2=s$ will be called {\it even}. Thus an even star-product gives rise to an even quantization. Conversely, an even quantization ${{\mathbf A}}$ of $A$ gives rise to an even star-product on~$A$.

\begin{Proposition} This defines a pair of mutually inverse bijections between even star-products on $A$ and ${\mathbb Z}/2$-equivariant quantizations ${{\mathbf A}}$ of $A$ equipped with a filtered antiautomorphism $\sigma\colon {{\mathbf A}}\to {{\mathbf A}}$ such that $\operatorname{gr}\sigma={\rm i}^d$ and $\sigma^2=s$ with a $\sigma$-invariant $($i.e., ${\mathbb Z}/4$-invariant$)$ quantization map.
\end{Proposition}

\begin{proof} Straightforward.
\end{proof}

\begin{Example}\label{MV} Let $A={\mathbb C}[x,y]$ with the Poisson bracket defined by $\lbrace y,x\rbrace=1$. Let ${{\mathbf A}}$ be the Weyl algebra generated by~$X$,~$Y$ with the defining relation $YX-XY=1$. This is a ${\mathbb Z}/2$-equivariant filtered quantization of $A$
with the filtration defined by $\deg(X)=\deg(Y)=1$. Define the map $\phi\colon A\to {{\mathbf A}}$ which sends $x^iy^j$ to the average of all the orderings of the monomial~$X^iY^j$; in other words, it is determined by the formula $\phi((px+qy)^n)=(pX+qY)^n$ for any $p,q\in {\mathbb C}$.
For example, $\phi(1)=1$, $\phi(x)=X$, $\phi(y)=Y$, $\phi(xy)=\frac{XY+YX}{2}=XY+\frac{1}{2}$. It is easy to show that this map gives rise to the symmetric Moyal--Weyl star-product
\begin{gather*}
a*b=\mu\left(\exp\left(\frac{\partial_y\otimes \partial_x-\partial_x\otimes \partial_y}{2}\right)(a\otimes b)\right),
\end{gather*}
where $\mu\colon A\otimes A\to A$ is the commutative multiplication. For example,
\begin{gather*}
x*y=\phi^{-1}(XY)=xy-\frac{1}{2}.
\end{gather*}
It is easy to check that this star-product is even, short, and nondegenerate.

The same statement holds in several variables, on a finite dimensional symplectic vector space~$V$ (i.e., for $A=SV$), where the symmetric Moyal--Weyl star-product has the form
\begin{gather*}
a*b=\mu(\exp(\pi/2)(a\otimes b)),
\end{gather*}
where $\pi\in \wedge^2V$ is the Poisson bivector.

More generally, given an element $\eta\in S^2V$, we may consider the short star-product
\begin{gather*}
a*b=\mu(\exp((\eta+\pi)/2)(a\otimes b)),
\end{gather*}
(the non-symmetric Moyal--Weyl star-product), which is not even if $\eta\ne 0$. Let $B\colon V\to V$ be the linear operator such that $\eta=(B\otimes 1)\pi$, i.e.,
\begin{gather}\label{nsmoyal}
a*b=\mu(\exp(((B+1)\otimes 1)\pi/2)(a\otimes b));
\end{gather}
then $B\in \mathfrak{sp}(V)$.
It is easy to show that $*$ is nondegenerate if and only if the operator $B+1$ (or, equivalently, $B-1$) is invertible.
\end{Example}

\subsection{Conical symplectic singularities}

The main class of examples of algebras $A$ we will be interested in is $A={\mathcal O}(X)$, where~$X$ is a (normal) conical symplectic singularity in the sense of Beauville (see~\cite{Ka, L2} and references therein for their definition and basic properties).

We will need the following lemma.

\begin{Lemma}[{\cite[Proposition~2.5]{L2}}]\label{inn} If $X$ is a conical symplectic singularity then every Poisson derivation of $A:=\mathcal{O}(X)$ is inner.
\end{Lemma}

\begin{proof} Here is another proof. Any such homogeneous derivation is defined by an (a priori multivalued) holomorphic Hamiltonian $H$ on the big symplectic leaf $X^\circ$ of $X$. Moreover, $H$ is, in fact, single-valued, since by~\cite{Na1} the algebraic fundamental group of $X^\circ$ is finite (indeed, we have a homomorphism $\phi\colon \pi_1(X^\circ,x_0)\to {\mathbb C}$ given by $\phi(\gamma)=\gamma(H)-H$, where $\gamma(H)$ is the analytic continuation of~$H$ along~$\gamma$, and $\phi=0$ since $\pi_1(X^\circ,x_0)_{\rm alg}$ is finite). Then by normality of~$X$ the function~$H$ extends to the whole~$X$, to a homogeneous holomorphic function on~$X$. This function is a holomorphic section of a line bundle on ${\mathbb P}X$, which is algebraic by the GAGA theorem, so $H\in \mathcal{O}(X)$, as claimed.
\end{proof}

It is shown in \cite{L2} that for conical symplectic singularities there is a very nice parametrization of filtered quantizations of $A$, which is the same as Namikawa's parametrization of filtered Poisson deformations~$X$~\cite{Na3}. Namely, they are parametrized by the finite dimensional space $\mathfrak{P}=H^2\big(\widetilde X^{\rm reg},{\mathbb C}\big)$, where $\widetilde X^{\rm reg}$ is the smooth part of the ${\mathbb Q}$-terminalization $\widetilde X$ of~$X$, modulo the action of a finite real reflection group~$W$ called the Namikawa Weyl group. Denote the quantization corresponding to $\lambda\in \mathfrak{P}$ by~${\mathbf A}_\lambda$. It is easy to show that ${\mathbf A}_\lambda$ is ${\mathbb Z}/2$-invariant.

Since $W$ is a reflection group, the space $\mathfrak{P}/W$ bijectively parametrizing quantizations is an affine space, whose coordinates are independent homogeneous generating invariants $p_1,\dots,p_r\in {\mathbb C}[\mathfrak{P}]^W$ of degrees $d_1,\dots,d_r$.

It is not hard to see that ${\mathbf A}_\lambda$ is even if and only if $\lambda$ and $-\lambda$ are equivalent with respect to the Namikawa Weyl group, i.e., if and only if there is an element $w\in W$ such that $w\lambda=-\lambda$ (see \cite{BPR}). This is equivalent to saying that $p_i=0$ whenever $d_i$ is odd, which defines a subspace in
$\mathfrak{P}/W$.

\begin{Remark}Lemma \ref{inn} implies that in the case of conical symplectic singularities the antiautomorphism $\sigma\colon {\mathbf A}\to {\mathbf A}$ such that $\sigma^2=s$ and $\operatorname{gr}\sigma={\rm i}^d$ is unique if exists. Indeed, if~$\sigma_1$,~$\sigma_2$ are two such antiautomorphisms then $g:=\sigma_1\circ \sigma_2^{-1}$ is an automorphism commuting with~$s$ such that $\operatorname{gr} g=1$. Hence $\log (g)$ is a linear combination of derivations of~${\mathbf A}$ of degree $\le -2$. If $g\ne 1$, then in the leading order $\log(g)$ yields a nonzero homogeneous derivation $D\colon A\to A$ of degree $\le -2$. By Lemma~\ref{inn}, it is inner, so given by a Hamiltonian of degree $\le 0$. But such a~Hamiltonian is necessarily constant, so $D=0$, a contradiction.

The same argument proves that if $g\colon {\mathbf A}\to {\mathbf A}$ is a filtration-preserving
automorphism commuting with $s$ then $g$ is completely determined by $\operatorname{gr}g$.
\end{Remark}

\subsection{HyperK\"ahler cones}

We will be especially interested in conical symplectic singularities which are {\it hyperK\"ahler cones}. Such singularities play an important role in physics, since many of them arise as Higgs and Coulomb branches in 3-dimensional $N=4$ superconformal field theories.

Namely, by a {\it hyperK\"ahler cone} we mean a normal conical Poisson variety~$X$ with Poisson bracket of negative degree (symplectic outside a set of codimension $\ge 2$) and a hyperK\"ahler structure on the smooth part that is compatible to the Poisson structure (i.e., the complex structure on $X$ is given by the operator $I$ and the Poisson structure by the symplectic form $\omega:=\omega_J+{\rm i}\omega_K$). Like smooth hyperK\"ahler manifolds, hyperK\"ahler cones are studied using their twistor space, which is the universal family of complex structures on $X$ fibered over ${\mathbb C}{\mathbb P}^1$. The literature (especially mathematical) on hyperK\"ahler cones is rather scarce, but more details can be found in \cite{Br, GNT,Sw}.\footnote{Note that the papers \cite{Br, Sw} only deal with the smooth part of~$X$. However, it is expected that the geometry of $X$ is fully controlled by what happens on the smooth part.}

An important property of hyperK\"ahler cones, which lies behind the phenomena described in this paper, is that the complex structures corresponding to all the points of ${\mathbb C}{\mathbb P}^1$ are equivalent, and in fact permuted by a group~${\rm SU}(2)$ acting real analytically on~$X$ (and by rotations on~\mbox{${\mathbb C}{\mathbb P}^1=S^2$}). Thus the twistor space of~$X$ has the form $({\rm SU}(2)\times X)/{\rm U}(1)$, where the action of~${\rm U}(1)$ on~$X$ comes from the grading on~$\mathcal O(X)$.

Important examples of hyperK\"ahler cones include closures of nilpotent orbits of a semisimple Lie algebra, Slodowy slices, quotients of a symplectic vector space by a finite group of symplectic automorphisms, and more generally hyperK\"ahler reductions of a symplectic vector space by a compact Lie group of symmetries (Higgs branches), which in particular include Nakajima quiver varieties. Coulomb branches (when they are conical) are also expected to be hyperK\"ahler cones~\cite{Nak}.

It is expected that under reasonable assumptions hyperK\"ahler cones and in particular Higgs and Coulomb branches should be conical symplectic singularities, and this is known for Nakajima quiver varieties~\cite{BS}. Also, sufficient conditions for a Higgs branch to be a symplectic singularity (which are satisfied in many examples) are given in~\cite{HSS}. So when we speak about hyperK\"ahler cones, we will assume that they are conical symplectic singularities. However, it must be noted that certain nilpotent orbit closures of classical groups fail to be symplectic singularities since they are not normal, even though they are sometimes Higgs branches~\cite{KP}.

\subsection{Finiteness of the number of fixed sets}

\begin{Lemma}\label{finfix} Let $X$ be an affine scheme of finite type over ${\mathbb C}$ and
$G$ be a reductive group acting on $X$. Then up to the $G$-action there are finitely many subschemes
of $X$ of the form $X^g$, where $g\in G$ is a semisimple element, .
\end{Lemma}

\begin{proof} First consider the case when $G$ is diagonalizable (in which case every $g\in G$ is semisimple). In this case the lemma is standard (see, e.g., \cite[Proposition on p.~44]{Ste}). Namely, pick generators $f_1,\dots,f_n$ of ${\mathbb C}[X]$ such that $f_i(gx)=\chi_i(g)f_i(x)$ for some characters~$\chi_i$ of~$G$. Then $f_i(gx)-f_i(x)=(\chi_i(g)-1)f_i(x)$, so~$X^g$ is cut out by the equations $f_i(x)=0$ for those~$i$ for which $\chi_i(g)\ne 1$, which implies that there are finitely many subschemes~$X^g$.

The general case reduces to the diagonalizable case by using the following fact (see~\cite{Sie} and \cite[Section~1.14]{Lu}): for each connected component~$Z$ of~$G$ there is a diagonalizable subgroup $D_Z\subset G$ such that every semisimple element $g\in Z$ is $G^0$-conjugate to an element of~$D_Z$.\footnote{We are grateful to G.~Lusztig for this explanation.}
\end{proof}

\section{Classification of short star-products}\label{section3}

In this section we present the classification of nondegenerate short star-products, the idea of which was explained to the first author by Maxim Kontsevich.

\subsection{Twisted traces}

Consider the setting of Section~\ref{constru}. Let $g\colon {{\mathbf A}}\to {{\mathbf A}}$ be a filtration-preserving invertible linear map, and let $T\colon {{\mathbf A}}\to {\mathbb C}$ be a $g$-twisted trace, i.e., $T({\mathbf a}{\mathbf b})=T({\mathbf b}g({\mathbf a}))$ for all ${\mathbf a},{\mathbf b}\in {{\mathbf A}}$. Note that taking ${\mathbf b}=1$, we get $T(g({\mathbf a}))=T({\mathbf a})$, i.e., $T$ is $g$-invariant.

Define the bilinear form $(\,,\,)_T$ on ${\mathbf A}$ by the formula $({\mathbf a},{\mathbf b})_T:=T({\mathbf a}{\mathbf b})$.

\begin{Lemma}\label{ide} Let $K$ be the left kernel of the bilinear form $(\,,\,)_T$. Then $K$ is also the right kernel of $(\,,\,)_T$, and a $g$-invariant two-sided ideal in~${{\mathbf A}}$. In particular, if ${\mathbf A}$ is a simple algebra and $T\ne 0$ then $K=0$.
\end{Lemma}

\begin{proof} Let $K'$ be the right kernel of $(\,,\,)_T$. We have ${\mathbf b}\in K'$ if and only if $T({\textbf{ab}})=0$ for all ${\mathbf a}\in {{\mathbf A}}$. This is equivalent to saying that $T({\mathbf b}g({\mathbf a}))=0$ for all ${\mathbf a}\in {{\mathbf A}}$, which happens if and only if ${\mathbf b}\in K$. Thus, $K'=K$. Also ${\mathbf a}\in K$ if and only if $T({\mathbf b}g({\mathbf a}))=0$ for all ${\mathbf b}\in {\mathbf A}$, which happens if and only if $g({\mathbf a})\in K$, hence $K$ is $g$-invariant. Finally, if ${\mathbf a}\in K$ and ${\mathbf b}\in {{\mathbf A}}$
then $({\textbf{ab}},{\textbf{c}})_T= ({\textbf{a}},{\textbf{bc}})_T=0$ and $({\textbf{c}},{\textbf{ba}})_T=({\textbf{cb}},{\textbf{a}})_T=0$, so ${\textbf{ab}},{\textbf{ba}}\in K$, i.e., $K$ is a two-sided ideal, as desired.
\end{proof}

By Lemma~\ref{ide}, it makes sense to speak of the kernel of $(\,,\,)_T$ (without specifying if it is left or right).

\begin{Lemma} If the kernel of $(\,,\,)_T$ vanishes then $g$ is an algebra automorphism.
\end{Lemma}

\begin{proof}We have
\begin{gather*}
T({\mathbf c}g({\mathbf a}{\mathbf b}))=T({\textbf{abc}})=T({\textbf{bc}}g({\mathbf a}))=T({\mathbf c}g({\mathbf a})g({\mathbf b})),
\end{gather*}
hence $g({\mathbf{ab}})=g({\mathbf a})g({\mathbf b})$.
\end{proof}

\subsection{Nondegenerate short star-products and nondegenerate twisted traces}

We will say that $T$ is {\it nondegenerate} if the restriction of the bilinear form $({\mathbf a},{\mathbf b})_T:=T({\mathbf a}{\mathbf b})$ to~$F_i{{\mathbf A}}$ is nondegenerate for each~$i$. Note that this is a strictly stronger condition than vanishing of the kernel of $(\,,\,)_T$ (for example, it implies that $T(1)\ne 0$).

Suppose that $T$ is nondegenerate and $s$-invariant (in this case $g$ commutes with~$s$). Then the bilinear form $(\,,\,)_T$ defines a canonical orthogonal complement ${\mathbf A}^T_i$ to $F_{i-1}{{\mathbf A}}$ in $F_i{{\mathbf A}}$. Note that the right and left orthogonal complement coincide since $g$ is filtration-preserving. We have a natural isomorphism $\theta_i^T\colon {\mathbf A}_i^T\to A_i$. Define the quantization map $\phi^T\colon A\to {{\mathbf A}}$ by $\phi^T|_{A_i}:=\big(\theta_i^T\big)^{-1}$. Clearly, $\phi^T$ is $s$-invariant. Thus, $\phi^T$ defines a star-product $*$ on $A$. Clearly, this star-product does not depend on the normalization of~$T$, so we will always normalize~$T$ so that $T(1)=1$ (which is possible since for a nondegenerate~$T$, one has $T(1)\ne 0$).

\begin{Proposition}[M.~Kontsevich]\label{kon} The star-product $*$ is short and nondegenerate. Moreover, any nondegenerate short star-product on~$A$ is obtained in this way from a unique nondegenerate~$T$ such that $T(1)=1$. In other words, this correspondence is a bijection between short nondegenerate star-products on $A$ and filtered quantizations ${{\mathbf A}}$ of~$A$ equipped with a filtration-preserving automorphism~$g$ commuting with $s$ and a nondegenerate $s$-invariant $g$-twisted trace~$T$ such that $T(1)=1$.
\end{Proposition}

\begin{proof} We have a nondegenerate (in general, nonsymmetric) inner product on~$A$ given by $\beta(a,b):=(\phi(a),\phi(b))_T$, and the spaces $A_i$ are orthogonal under this inner product. Moreover, it is clear that $\beta(a*b,c)=\beta(a,b*c)$. Suppose $a\in A_m$, $b\in A_n$. Let us show that every homogeneous component of $a*b$ has degree $\ge |m-n|$. First assume that $m>n$. It suffices to show that for any homogeneous $c\in A$ of degree $d<m-n$, one has $\beta(a*b,c)=0$. But we have $\beta(a*b,c)=\beta(a,b*c)$, and $b*c$ does not have any terms of degree bigger than $n+d<m$.
Hence $\beta(a*b,c)=\beta(a,b*c)=0$, as desired. A similar argument applies if $m<n$: we have $\beta(c,a*b)=\beta(c*a,b)=0$. Thus, the star-product~$*$ is short. Moreover, we have
\begin{gather*}
\langle a,b\rangle={\rm CT}(a*b)=\beta(a*b,1)=\beta(a,b),
\end{gather*}
which implies that $\langle\,,\,\rangle$ (and hence~$*$) is nondegenerate.

Conversely, assume that $*$ is a nondegenerate short star-product on~$A$. Then the quantization~${{\mathbf A}}$ is identified with $(A,*)$ using the quantization map~$\phi$. Set $T({\mathbf a}):={\rm CT}\big(\phi^{-1}({\mathbf a})\big)$. Then the form $({\textbf{a}},{\textbf{b}})_T:=T({\textbf{ab}})$ is nondegenerate on each $F_i{{\mathbf A}}$. Thus there exists a unique linear automorphism~$g_i$ of $F_i{{\mathbf A}}$ such that $T({\textbf{ab}})=T({\mathbf b}g_i({\mathbf a}))$. Moreover, it is clear that~$g_i$ preserves~$F_{i-1}{{\mathbf A}}$, and $g_i|_{F_{i-1}{{\mathbf A}}}=g_{i-1}$. Thus, the maps~$g_i$ define a~filtration-preserving linear automorphism $g\colon {{\mathbf A}}\to {{\mathbf A}}$.

Finally, it is easy to check that these assignments are mutually inverse, as claimed.
\end{proof}

\begin{Remark} If $T$ is any $s$-invariant linear functional on ${\mathbf A}$ such that the form $({\mathbf a},{\mathbf b})_T=T({\mathbf a}{\mathbf b})$ is nondegenerate on each $F_i{\mathbf A}$ (which clearly happens for ``random'' $T$) then for each~$i$ we can define a linear automorphism $g_i$ of $F_i({\mathbf A})$ such that $T({\textbf{ab}})=T({\mathbf b}g_i({\mathbf a}))$. However, in general $g_i|_{F_{i-1}{\mathbf A}}\ne g_{i-1}$ and $g_i$ does not preserve~$F_{i-1}{\mathbf A}$ (so the automorphism~$g$ of ${\mathbf A}$ restricting to~$g_i$ on each $F_i{\mathbf A}$ is not defined). Thus, the left and right orthogonal complements of~$F_{i-1}{\mathbf A}$ in~$F_i{\mathbf A}$ do not coincide, and one cannot define a quantization map giving a short star-product. The condition that~$g$ is well defined is a very strong restriction of~$T$, and we will see that it usually leads to a finite dimensional space of possible~$T$.
\end{Remark}

Let ${\mathbf A}g$ be the ${\mathbf A}$-bimodule which is ${\mathbf A}$ as a left module and ${\mathbf A}$ with action twisted by $g$ as a right module.

\begin{Corollary}\label{bun} Let $S=S({\mathbf A})$ be the set of nondegenerate short star-products corresponding to a filtered quantization ${{\mathbf A}}$ of $A$. Then we have a map $\pi\colon S\to \operatorname{Aut}({{\mathbf A}})$ from $S$ to the group of filtration-preserving $s$-invariant automorphisms of ${{\mathbf A}}$ whose fiber $\pi^{-1}(g)$ is a subset of the space $(HH_0({{\mathbf A}},{{\mathbf A}}g)^s)^*$ dual to the $s$-invariants in the Hochschild homology $HH_0({{\mathbf A}},{{\mathbf A}}g)$.
\end{Corollary}

\begin{proof} This follows immediately from Proposition~\ref{kon}, since by definition, an $s$-invariant $g$-twisted trace on ${{\mathbf A}}$ is an element of $(HH_0({{\mathbf A}},{{\mathbf A}}g)^s)^*$.
\end{proof}

\begin{Corollary} \label{bun1} Suppose $A=\mathcal{O}(X)$, where $X$ is a conical symplectic
singularity. Then any nondegenerate short star-product corresponding to $g\in \operatorname{Aut}({{\mathbf A}})$
is invariant with respect to the connected component~$Z_g^0$ of the centralizer~$Z_g$
of~$g$ in $\operatorname{Aut}({{\mathbf A}})$. In particular, any nondegenerate even short star-product
is invariant under $\operatorname{Aut}({\mathbf A})^0$.
\end{Corollary}

\begin{proof} It suffices to show that the Lie algebra $\operatorname{Lie}Z_g$ acts trivially on
$HH_0({\mathbf A},{\mathbf A} g)$. Let $z\in \operatorname{Lie}Z_g$. By Lemma~\ref{inn}, there exists a unique ${\mathbf z}\in F_2{\mathbf A}$ up to adding a constant such that $z({\mathbf a})=[{\mathbf z},{\mathbf a}]$ for all ${\mathbf a}\in {\mathbf A}$. Then $g({\mathbf z})={\mathbf z}+C$ for some constant $C$, so $C=g({\mathbf z})-{\mathbf z}$. Thus if $C\ne 0$ then for any $g$-twisted trace $T$ we have $T(1)=0$, so there are no nondegenerate short star-products at all. On the other hand, if $C=0$ then we have $[{\mathbf z},{\mathbf a}]={\mathbf z}{\mathbf a}-{\mathbf a}g({\mathbf z})$, so ${\mathbf z}$ acts trivially on $HH_0({\mathbf A},{\mathbf A} g)$, as claimed. The last statement follows from the fact that in the even case $g=s$, and~$s$ by definition is central in~$\operatorname{Aut}({{\mathbf A}})$.
\end{proof}

\subsection{The even case}

\begin{Proposition}\label{criteven} A short nondegenerate star-product~$*$ on~$A$ is even if and only if the corresponding automorphism $g$ of the associated even quantization ${\mathbf A}$ of $A$ equals $s$ and the corresponding nondegenerate trace $T$ on ${\mathbf A}$ is $\sigma$-invariant.
\end{Proposition}

\begin{proof} Suppose $a,b\in A$, $\deg(a)=\deg(b)=d$. Then if $*$ is even then
\begin{gather*}
{\rm CT}(a*b)=C_d(a,b)=(-1)^dC_d(b,a)=C_d\big(b,(-1)^da\big)={\rm CT}\big(b*(-1)^da\big),
\end{gather*}
hence $g=(-1)^d=s$. Also $T$ is $\sigma$-invariant since so is $\phi$.

Conversely, suppose $g=s$ and $T$ is $\sigma$-invariant. Let $a\in A_m$, $b\in A_n$, $c\in A_{m+n-2k}$.
Let $\phi(a)={\mathbf a}$, $\phi(b)={\mathbf b}$, $\phi(c)={\mathbf c}$. We have
\begin{gather*}
T(\phi(a*b){\mathbf c})=T({\mathbf a}{\mathbf b}{\mathbf c})=T(\sigma({\mathbf a}{\mathbf b}{\mathbf c}))=
T(\sigma({\mathbf c})\sigma({\mathbf b})\sigma({\mathbf a}))\\
\hphantom{T(\phi(a*b){\mathbf c})}{}
=T\big(\sigma({\mathbf b})\sigma({\mathbf a})\sigma^{-1}({\mathbf c})\big)=(-1)^kT({\mathbf b}{\mathbf a}{\mathbf c})=T\big(\phi\big((-1)^kb*a\big){\mathbf c}\big).
\end{gather*}
Thus $C_k(a,b)=(-1)^kC_k(b,a)$, i.e., $*$ is even.
\end{proof}

Let ${{\mathbf A}}_\pm$ be the $\pm 1$-eigenspaces of $s$ on ${{\mathbf A}}$.

\begin{Corollary}\label{criteven1} A star-product $*$ is even if and only if $T$ is $\sigma$-stable
and defines an element of $(HH_0({{\mathbf A}},{{\mathbf A}}s)^s)^*$, where
 \begin{gather*}
 HH_0({{\mathbf A}},{{\mathbf A}}s)^s=({{\mathbf A}}s/[{{\mathbf A}},{{\mathbf A}}s])^s=
 {{\mathbf A}}_+/([{{\mathbf A}}_+,{{\mathbf A}}_+]+\lbrace{{{\mathbf A}}_-,{{\mathbf A}}_-\rbrace}),
 \end{gather*}
 and $\lbrace{{\mathbf a},{\mathbf b}\rbrace}:={\mathbf{ab}}+{\mathbf{ba}}$.
\end{Corollary}

\begin{proof} This follows immediately from Corollary \ref{bun} and Proposition \ref{criteven}.
\end{proof}

\subsection{Examples}

\begin{Example}\label{repth1} Let $V$ be a finite dimensional symplectic vector space. Let us classify nondegenerate short star-products for the polynomial algebra $A=SV$ with the standard grading and Poisson structure. The quantization ${{\mathbf A}}$ in this case is the Weyl algebra ${\mathbf{W}}(V)$. Its group of filtered automorphisms commuting with $s$ is $G={\rm Sp}(V)$.

Let $g\in G$, and let us compute the space ${\mathcal T}:=HH_0({{\mathbf A}},{{\mathbf A}}g)^*$. First, we claim that if $g$ has eigenvalue $1$ then this space vanishes. Indeed, let $u\in V$ be a nonzero vector such that $g u=u$. Then for $T\in {\mathcal T}$ we have $T([u,{\mathbf b}])=0$ for all ${\mathbf b}\in {\mathbf A}$. It is easy to see that any element ${\mathbf a}\in {{\mathbf A}}$ can be written as $[u,{\mathbf b}]$. Thus, $T=0$, as desired.

Now consider the case when $g$ has no eigenvalue $1$. In this case, $T({\mathbf{ab}}-{\mathbf b}g({\mathbf a}))=0$. Let $E\subset {{\mathbf A}}$ be the span of elements ${\mathbf{ab}}-{\mathbf b}g({\mathbf a})$ for any ${\mathbf a},{\mathbf b}\in {\mathbf A}$. Taking $a,b\in A$ homogeneous and ${\mathbf a}$, ${\mathbf b}$ their lifts to ${\mathbf A}$, we see that $\operatorname{gr}E$ contains elements $(a-g(a))b$, where $a,b\in A$ are homogeneous. Taking $\deg(a)=1$, we find that $\operatorname{gr}E$ contains the augmentation ideal of~$SV$, which shows that $\dim {\mathcal T}\le 1$.

Let us show that in fact $\dim {\mathcal T}=1$. For this purpose it suffices to produce a nonzero $g$-twisted trace. To this end, let $B:=(1+g)(1-g)^{-1}$ be the (symplectic) Cayley transform of $g$; then $B\in {\mathfrak{sp}}(V)$ and $g=(B+1)^{-1}(B-1)$, so $B+1$ and $B-1$ are invertible. Hence we have the
Moyal--Weyl product given by~\eqref{nsmoyal}, which is short and nondegenerate. As we have shown, such a star-product defines a nondegenerate $g$-twisted trace~$T$. Thus we have $\mathcal{T}=\mathcal{T}^s={\mathbb C}T$, a~1-dimensional space, i.e., we have a unique star-product for each $g$.

In summary, we see that short nondegenerate star-products on $SV$ are exactly the Moyal--Weyl products parametrized by elements of $g\in {\rm Sp}(V)$ without eigenvalue $1$, or, equivalently, elements of $B\in {\mathfrak{sp}}(V)$ without eigenvalues $1$ and $-1$. The two parametrizations are related by the Cayley transform. The even case corresponds to $g=-1$, i.e., $B=0$, i.e., the symmetric Moyal--Weyl product is the unique even short nondegenerate star-product.

Here is an explicit construction of the corresponding twisted trace. Assume first that $g$ has no eigenvalues of absolute value~$1$. Let $L\subset V$ be the direct sum of generalized eigenspaces of $g$ with eigenvalues $\lambda$ such that $|\lambda|<1$. Then the Weyl algebra~${\mathbf A}$ is naturally identified with the algebra $D(L^*)$ of differential operators on~$L^*$, which acts naturally on $M:=SL={\mathbb C}[L^*]$. Then the trace~$T$ can be defined by the formula
\begin{gather*}
T({\mathbf a})=\frac{\operatorname{Tr}_M({\mathbf a} g)}{\operatorname{Tr}_M(g)}={\det}_L(1-g)\operatorname{Tr}_M({\mathbf a} g).
\end{gather*}
Note that the trace is convergent since the eigenvalues of~$g$ on~$L$ have absolute value~$<1$. Moreover, it is easy to check that the right hand side is a rational function of~$g$ which is regular on the locus where~$g$ has no eigenvalue~$1$. Thus the above formula understood as a rational function of~$g$ gives the desired twisted trace in the general case.
\end{Example}

\begin{Example}\label{repth}Let ${{\mathbf A}}=D_\lambda:={\rm U}(\mathfrak{sl}_2)/I_\lambda$, where
$I_\lambda$ is the ideal generated by $C-\frac{\lambda(\lambda+2)}{2}$, where $C=ef+fe+h^2/2$ is the Casimir. Then $A$ is the algebra of regular functions on the quadratic cone $X\subset {\mathbb C}^3$,
$A={\mathbb C}[x,y,z]/\big(xy-z^2\big)$, with the variables having degree~$2$. Let us classify short nondegenerate star-products giving this quantization.

Let us start with the even case. We have $s=1$ and ${\mathbf A}={\mathbf A}_+$, so by Corollary~\ref{criteven1}, the star-product gives rise to a trace in $HH_0({\mathbf A},{\mathbf A})^*$, which is well known to be 1-dimensional. This trace is therefore $\sigma$-stable and also nondegenerate for Weil generic $\lambda$ (i.e., outside of a countable set), since it is nondegenerate for $\lambda=-1/2$,
when $D_\lambda$ is isomorphic to the even part of the Weyl algebra, ${\mathbf A}_1^{{\mathbb Z}/2}$, so we can use the symmetric Moyal--Weyl product. This nondegenerate trace gives rise to the unique ${\rm PGL}(2)$-invariant short star-product, which is thus even and nondegenerate.

However, for special values of $\lambda$ (namely, $\lambda\in {\mathbb Z}$, $\lambda\ne -1$) this trace is degenerate. Indeed, we may assume $\lambda\in {\mathbb Z}_{\ge 0}$, then the trace is given by the character of the finite dimensional irreducible representation $L_\lambda$ of ${\mathfrak{sl}}_2$ with highest weight $\lambda$, and the annihilator of this representation is the kernel of $(\,,\,)_T$. The short even star-product still exists at these points (by continuity), namely it comes from the unique
${\rm PGL}(2)$-invariant quantization map, but it is degenerate.

In fact, it is not hard to see that outside of these special points, the star-product is nondegenerate.
Indeed, it is clear that the bilinear form $C_{2m}$ on $A_{2m}$ is the standard invariant pairing
$V_{2m}\otimes V_{2m}\to {\mathbb C}$ times a polynomial $P_m(\lambda)$ of degree~$2m$.
Moreover, as explained above, this polynomial vanishes at $0,\dots,m-1$ and $-2,\dots,-m-1$,
so it is a nonzero multiple of $\prod\limits_{j=0}^{m-1}(\lambda-j)(\lambda+j+2)$, i.e.,
it does not vanish at any other points, as claimed.

Now let $g\ne 1$. We have $\operatorname{Aut}({\mathbf A})={\rm PGL}_2({\mathbb C})$, so $g\in {\rm PGL}_2({\mathbb C})$. First consider the case when~$g$ is semisimple (so we may assume that~$g$ is in the standard torus).
Assume that $|\alpha(g)|>1$ for the positive root~$\alpha$. Let $M_\lambda$ be the Verma module
over $\mathfrak{sl}_2$ with highest weight~$\lambda$, which is also a $D_\lambda$-module. Consider the linear functional $T_\lambda^w$ on $D_\lambda$ given by
\begin{gather*}
T_\lambda^w({\mathbf a}):=\frac{\operatorname{Tr}|_{M_\lambda}({\mathbf a} g)}{\operatorname{Tr}|_{M_\lambda}(g)}=
(1-w)\sum_{n\ge 0}(v_n^*,{\mathbf a} v_n)w^n,
\end{gather*}
where $w=\alpha(g)^{-1}$, $v_n=f^nv$ is a homogeneous basis of $M_\lambda$ (where $v$ is a highest weight vector) and $v_n^*$ is the dual basis of the restricted dual space $M_\lambda^\vee$. It is easy to show that this series is convergent to a rational function of $w$ regular for $w\ne 1$, and thus defines a $g$-twisted trace on~$D_\lambda$. We also have the $g$-twisted trace $T_{-\lambda-2}^w$, since by definition
$D_\lambda=D_{-\lambda-2}$. Moreover, we have $\mathcal{O}(X^g)={\mathbb C}[z]/\big(z^2\big)$, so by Proposition~\ref{finde1} below, the space of $g$-twisted traces $HH_0({\mathbf A},{\mathbf A}g)^*$ is at most 2-dimensional. On the other hand, for $\lambda\ne -1$, the traces $T_\lambda^w$ and $T_{-\lambda-2}^w$ are linearly independent, since $T_\lambda^w(1)=1$ and $T_\lambda^w(h)=\lambda-\frac{2w}{1-w}$, thus they form a basis in the space of $g$-twisted traces. Moreover, at $\lambda=-1$ we have the trace
$\big(T_{-1}^w\big)':=\lim\limits_{\varepsilon\to 0}\frac{T_{-1+\varepsilon}^w-T_{-1-\varepsilon}^w}{\varepsilon}$, and
it is easy to see that $T_{-1}^w$, $\big(T_{-1}^w\big)'$ form a basis in the space of traces
when $\lambda=-1$.

Moreover, for $\lambda\ne -1$ the trace $T_\lambda$ corresponding $g=1$ can be obtained by the formula
\begin{gather*}
T_\lambda({\mathbf a})=\frac{1}{\lambda+1}\lim_{w\to 1}\frac{T_\lambda^w({\mathbf a})-w^{\lambda+1}T_{-\lambda-2}^w({\mathbf a})}{1-w}.
\end{gather*}
Indeed, if $\lambda$ is a nonnegative integer, this equals $\frac{1}{\lambda+1}\operatorname{Tr}|_{L_\lambda}({\mathbf a})$. The case $\lambda=-1$ is then obtained as a limit, using that this is in fact a polynomial in $\lambda$. For example, for ${\mathbf a}=h^2$ we get
\begin{gather*}
\operatorname{Tr}|_{L_\lambda}\big(h^2\big)=2\sum_{0\le j\le \lambda/2}(\lambda-2j)^2=\frac{\lambda(\lambda+1)(\lambda+2)}{3}.
\end{gather*}
Thus
\begin{gather*}
T_\lambda\big(h^2\big)=\frac{\lambda(\lambda+2)}{3}.
\end{gather*}

It is easy to show that Weil generically\footnote{Here ``Weil generically'' means ``outside of a countable union of algebraic hypersurfaces in the parameter space''.} these traces are nondegenerate. Thus we obtain a~3-parameter family of nondegenerate $h$-invariant short star-products on $A$ inequivalent under ${\rm PGL}_2({\mathbb C})$, and any nondegenerate short star-product with semisimple $g$ is equivalent to one of them. Namely, the parameters are $\lambda$, $w$, and $\alpha:=T(h)$ (where $T$ is normalized, i.e., $T(1)=1$).

Finally, it remains to consider the case when $g$ is not semisimple (i.e., a Jordan block). In this case, by Proposition~\ref{finde2} below, the space of traces is at most 1-dimensional. If~$\lambda$ is a~nonnegative integer, the corresponding trace can be given by the formula
\begin{gather*}
\widetilde{T}_\lambda({\mathbf a})=\frac{1}{\lambda+1}\operatorname{Tr}|_{L_\lambda}({\mathbf a} g)
\end{gather*}
(where, abusing notation, we denote by $g$ the unipotent lift of $g$ to ${\rm SL}_2({\mathbb C})$). It is not hard to check that this is in fact a polynomial of~$\lambda$, so it can be naturally extended to any value of~$\lambda$.
\end{Example}

\subsection{Finite dimensionality of the space of nondegenerate short star-products}

\subsubsection{The case of a semisimple automorphism}

\begin{Proposition}\label{finde1} Let $A$ be a finitely generated Poisson algebra such that the Poisson scheme $X=\operatorname{Spec}A$ has finitely many symplectic leaves. Let~${\mathbf A}$ be a quantization of~$A$, and \mbox{$g\colon {{\mathbf A}}\to {{\mathbf A}}$} be a semisimple automorphism which commutes with~$s$. Then nondegenerate short star-products on~$A$, if they exist, are parametrized by points of the finite dimensional vector space \linebreak $(HH_0({{\mathbf A}},{{\mathbf A}}g)^s)^*$ which don't belong to a countable collection of hypersurfaces, modulo scaling.

Moreover, $\dim HH_0({{\mathbf A}},{{\mathbf A}}g)^s\le \dim HP_0(X^g)^s<\infty$, where $X^g$ is the fixed point scheme of~$g$ on~$X$, and $HP_0(X^g):= \mathcal{O}(X^g)/\lbrace{\mathcal{O}(X^g), \mathcal{O}(X^g)\rbrace}$ is the zeroth Poisson homology of $X^g$.\footnote{Here, abusing notation, we denote the associated graded isomorphism $\operatorname{gr}g\colon A\to A$ simply by $g$.}
\end{Proposition}

\begin{proof} Let $E$ be the span of elements ${\mathbf{ab}}-{\mathbf{b}}g({\mathbf a})$ for ${\mathbf{a}},{\mathbf{b}}\in {{\mathbf A}}$.

We claim that for any $a,b\in A$, $\operatorname{gr}(E)$ contains the element $(a-g(a))b$. Indeed, we may assume that $a$, $b$ are homogeneous of degrees $i$, $j$, and let ${{\mathbf a}}$, ${\mathbf b}$ be their lifts to ${{\mathbf A}}$. Then the leading term (of degree $i+j$) of ${{\mathbf a}} {\mathbf b}-{\mathbf b}g({{\mathbf a}})$ is $(a-g(a))b$, as desired.

\looseness=-1 Also, we claim that if $a\in A$ and $g(a)=a$ then for any $b\in A$, $\lbrace{a,b\rbrace}$ belongs to $\operatorname{gr}(E)$. Indeed, we can again assume that $a$, $b$ are homogeneous of degrees $i$, $j$. Since $g=1$ on its generalized 1-eigenspace, we can choose a lift ${\mathbf a}$ of~$a$ to~${\mathbf A}$ which is invariant under~$g$. Then choosing any lift~${\mathbf b}$ of~$b$ to~${\mathbf A}$, we see that the leading term (of degree $i+j-2$) of ${{\mathbf a}} {\mathbf b}-{\mathbf b}g({{\mathbf a}})=[{{\mathbf a}},{\mathbf b}]$ is $\lbrace{ a,b\rbrace}$, as desired.

Now let $I\subset A$ be the span of elements $(a-g(a))b$. This is the ideal of the fixed point subscheme~$X^g$. Note that $I\cap A^g$ is a Poisson ideal in $A^g$. Hence, $A/I=A^g/(I\cap A^g)$ is a Poisson algebra, so~$X^g$ is a Poisson scheme. Moreover, since $g=1$ on its generalized 1-eigenspace, the $g$-invariants in the tangent space to the symplectic leaf~$L$ at any point $P\in X^g\cap L$ form a~nondegenerate subspace, hence $X^g\cap L$ is symplectic. Thus, $X^g$ also has finitely many symplectic leaves, which are connected components of intersections of $X^g\cap L$ for various~$L$.

 Now, the space $A/\operatorname{gr}(E)$ receives a surjective $s$-invariant map from the space $HP_0(X^g)=\mathcal{O}(X^g)/\lbrace{ \mathcal{O}(X^g),\mathcal{O}(X^g) \rbrace}$. This space is finite dimensional by \cite[Theorem~1.4]{ES}. Hence the space $A/\operatorname{gr}(E)$
 is finite dimensional. Therefore the space ${{\mathbf A}}/E=HH_0({{\mathbf A}},{{\mathbf A}}g)$ is finite dimensional, and we have the inequalities
\begin{gather*}
\dim HH_0({{\mathbf A}},{{\mathbf A}}g)\le \dim HP_0(X^g),\qquad \dim HH_0({{\mathbf A}},{{\mathbf A}}g)^s\le \dim HP_0(X^g)^s.
\end{gather*}

\looseness=-2 Now the proposition follows, since by Corollary~\ref{bun}, short nondegenerate star-products corre\-spon\-ding to~$g$ are parametrized by the subset of $(HH_0({{\mathbf A}},{{\mathbf A}}g)^s)^*$ determined by the nondegenera\-cy condition, which is the vanishing condition for countably many determinant polynomials.
\end{proof}

\begin{Remark} \quad
\begin{enumerate}\itemsep=0pt
\item It is clear from the proof of Proposition~\ref{finde1} that it holds more generally, namely when~$g$ acts trivially on its generalized 1-eigenspace (this is the only consequence of semisimplicity of~$g$ used in the proof).
\item In particular, Proposition \ref{finde1} applies to the setting of \cite{L2}, namely $A=\mathcal{O}(X)$, where~$X$ is a conical symplectic singularity in the sense of Beauville. Indeed, in this case it is known that~$X$ has finitely many symplectic leaves \cite[Theorem~2.3]{Ka}. This property also holds for Higgs branches considered in~\cite{BPR} (by \cite[Section~2.3]{L3}), and is expected for Coulomb branches which are cones, so Proposition~\ref{finde1} applies to all such cases.
\end{enumerate}
\end{Remark}

\begin{Corollary}\label{finde} Let $A$ be a finitely generated algebra, and the Poisson scheme $X=\operatorname{Spec}A$ have finitely many symplectic leaves. Then nondegenerate even short star-products
on $A$ defining a~given even quantization ${{\mathbf A}}$, if they exist, are parametrized by points of
the finite dimensional vector space $(HH_0({{\mathbf A}},{{\mathbf A}}s)^\sigma)^*$ which don't belong to a countable collection of hypersurfaces, modu\-lo scaling.

Moreover, $\dim HH_0({{\mathbf A}},{{\mathbf A}}s)^\sigma\le \dim HP_0(X^s)^\sigma<\infty$.
\end{Corollary}

\begin{proof} This immediately follows from Proposition~\ref{finde1}, specializing to the case $g=s$ and restricting to $\sigma$-invariant traces using Corollary~\ref{criteven1}.
\end{proof}

Corollary~\ref{finde} shows that if $\operatorname{Spec}A$ has finitely many symplectic leaves
then nondegenerate short even star-products depend on finitely many parameters, as conjectured in~\cite{BPR}.

\subsubsection{The general case}
Now let us generalize Proposition \ref{finde1} to the nonsemisimple case. For this purpose assume that every Poisson derivation of $A$ of nonpositive degree is inner. Then the Lie algebra $\g$ of the group
$G:=\operatorname{Aut}_0(A)$ of filtration-preserving Poisson automorphisms of $A$ commuting with $s$ has the form $\g=A_2$, where the operation is the Poisson bracket.

Let us also assume that the group~$G$ is reductive. Then it is easy to see that the action of $\g$ on $A$ uniquely lifts to any quantization ${\mathbf A}$, so we have ${\rm Der}({\mathbf A})^s=\g$. Thus we have a Lie algebra inclusion $\g\subset F_2{\mathbf A}\subset {\mathbf A}$.

Let $g\in G$. Write $g=g_0\exp(e)=\exp(e)g_0$ where $g_0\in G$ is semisimple and $e\in \g$ is nilpotent.

\begin{Proposition}\label{finde2} Under the above assumptions, the conclusion of Proposition~{\rm \ref{finde1}} holds for~$g$. Moreover,
\begin{gather*}
\dim HH_0({\mathbf A},{\mathbf A}g)^s\le \dim HH_0({\mathbf A},{\mathbf A}g_0)^s\le \dim HP_0(X^{g_0})^s<\infty.
\end{gather*}
\end{Proposition}

\begin{proof} Let $Z$ be the centralizer of $g_0$ in $G$, a reductive group containing $\exp(e)$. By the Jacobson--Morozov lemma, the Lie algebra $\mathfrak{z}$ of $Z$ contains an $\mathfrak{sl}_2$-triple
$e$, $h$, $f$. We have a~decomposition ${\mathbf A}=\oplus_{j\in {\mathbb Z}}{\mathbf A}_j$, where ${\mathbf A}_j$ is the $j$-eigenspace of~$h$. Since $g(e)=e$, we have $e{\mathbf b}-{\mathbf b}g(e)=[e,{\mathbf b}]$, so for any $T\in HH_0({\mathbf A},{\mathbf A}g)^*$ and ${\mathbf b}\in {\mathbf A}$ we have $T([e,{\mathbf b}])=0$. By representation theory of $\mathfrak{sl}_2$, this implies that $T({\mathbf a})=0$ for any ${\mathbf a}\in {\mathbf A}_j$ with $j>0$.

Now consider the 1-parameter subgroup $t^h$ of $G$ ($t\in {\mathbb C}^\times$), and define
$T_t({\mathbf a}):=T\big(t^{-h}{\mathbf a}t^h\big)$. If ${\mathbf a}\in {\mathbf A}_j$ then $T_t({\mathbf a})=t^{-j}T({\mathbf a})$. Thus, since $T$ vanishes on ${\mathbf A}_j$ for $j>0$, there exists a limit $T_0=\lim\limits_{t\to 0}T_t$, which coincides with $T$ on ${\mathbf A}_0$ and vanishes on ${\mathbf A}_j$ for $j\ne 0$.

We have $T({\mathbf a})=T_t({\mathbf a}_t)$, where ${\mathbf a}_t:=t^h{\mathbf a} t^{-h}$.
Recall that
\begin{gather*}
{\mathbf a}{\mathbf b}-{\mathbf b} g({\mathbf a})={\mathbf a}{\mathbf b}-{\mathbf b}g_0({\mathbf a})-{\mathbf b}\sum_{i\ge 1}\frac{g_0e^i}{i!}({\mathbf a})\in \operatorname{Ker}(T).
\end{gather*}
Hence
\begin{gather*}
t^h({\mathbf a}{\mathbf b}-{\mathbf b} g({\mathbf a}))t^{-h}={\mathbf a}_t{\mathbf b}_t-{\mathbf b}_t g_0({\mathbf a}_t)-{\mathbf b}_t\sum_{i\ge 1}\frac{t^{2i}e^i}{i!}({\mathbf a}_t)\in \operatorname{Ker}(T_t).
\end{gather*}
Thus, for any ${\mathbf a},{\mathbf b}\in {\mathbf A}$ the element ${\mathbf a}{\mathbf b}- {\mathbf b} g_0({\mathbf a})-{\mathbf b}\sum\limits_{i\ge 1}\frac{t^{2i}g_0e^i}{i!}({\mathbf a})$ belongs to $\operatorname{Ker}(T_t)$. Sending~$t$ to zero, we see that ${\mathbf a}{\mathbf b}-{\mathbf b} g_0({\mathbf a})\in \operatorname{Ker}(T_0)$. Thus, $T_0\in (HH_0({\mathbf A},{\mathbf A}g_0)^s)^*$, which is finite dimensional (with dimension dominated by $\dim HP_0(X^{g_0})^s$) by Proposition~\ref{finde1}.

It remains to show that the assignment $T\mapsto T_0$ is injective; this implies both statements of the theorem. To this end, assume that $T_0=0$ but $T\ne 0$, and let $m>0$ be the order of vanishing of $T_t$ at $t=0$. Let $T_0':=\lim\limits_{t\to 0}T_t/t^m$. Then $T_0'$ is a nonzero element of $HH_0({\mathbf A},{\mathbf A}g_0)^*$ which vanishes on ${\mathbf A}_0$. But if ${\mathbf a}\in {\mathbf A}_j$, $j\ne 0$, then ${\mathbf a}=\frac{1}{j}[h,{\mathbf a}]$, so $T_0'({\mathbf a})=0$. This is a~contradiction, which completes the proof of the proposition.
\end{proof}

\begin{Theorem}\label{finde3} Let $A=\mathcal{O}(X)$, where $X$ is a conical symplectic singularity
with Poisson bracket of degree $-2$ such that the group $G$ of dilation-invariant Poisson automorphisms of $X$ is reductive. Then for any quantization ${\mathbf A}$ of $A$ and any $g\in G$ the space $HH_0({\mathbf A},{\mathbf A}g)^s$ is finite dimensional, with dimension dominated by $\dim HP_0(X^{g_0})^s$, which is finite. Therefore, nondegenerate short star-products on~$A$ depend on finitely many parameters.
\end{Theorem}

\begin{proof}The theorem follows immediately from the above results and Lemmas \ref{inn}, \ref{finfix}.
\end{proof}

The group $G$ is not always reductive, see the example in the beginning of~\cite{Na2}. However, this is true in most applications, in particular if $X$ is a hyperK\"ahler cone (the case considered in~\cite{BPR} and~\cite{DPY}), i.e., the holomorphic symplectic structure on the big symplectic leaf on $X$ comes from a~hyperK\"ahler structure (as complex structure~$I$). In this case, by \cite[Theorem~3.3]{HKLR} (see also the discussion in \cite[p.~22]{GNT}), the hyperK\"ahler metric is determined by the holomorphic symplectic structure and the complex conjugation map through the twistor construction (possibly up to finitely many choices). This means that the real part $G^0_{{\mathbb R}}$ of the identity component $G^0$ of the group $G$ preserves the hyperK\"ahler metric. This implies that $G^0_{{\mathbb R}}$ preserves the unitary structure on the space~$A_d$ for each~$d$ given by
\begin{gather*}
||f||^2:=\int_X |f(z)|^2{\rm e}^{-D(z)}{\rm d}V,
\end{gather*}
where ${\rm d}V$ is the volume element and $D(z)$ the distance to the origin determined by the hyperK\"ahler metric. This implies that $G^0_{{\mathbb R}}$ is compact, hence~$G^0$ and~$G$ are reductive, as desired.

Thus we obtain the following generalization of the conjecture from~\cite{BPR} to the not necessarily even case, which is our first main result.

\begin{Theorem}\label{main1}Theorem~{\rm \ref{finde3}} holds for hyperK\"ahler cones.
\end{Theorem}

\subsection{Conjugations}

As before, let $A$ be a graded Poisson algebra with Poisson bracket of degree $-2$.

\begin{Definition} \quad
\begin{enumerate}\itemsep=0pt
\item[(i)] A {\it conjugation} on $A$ is a ${\mathbb C}$-antilinear degree-preserving Poisson automorphism $\rho\colon A\to A$.
\item[(ii)] If ${\mathbf A}$ is a quantization of $A$ then a conjugation on~${\mathbf A}$ is an antilinear filtration-preserving automorphism~$\rho$ of~${\mathbf A}$ which commutes with~$s$.
\end{enumerate}
\end{Definition}

It is clear that any conjugation on ${\mathbf A}$ gives rise to a conjugation on~$A$ (by taking the associated graded). Also note that if~$\rho$ is a conjugation then~$\rho^2$ is a~${\mathbb C}$-linear automorphism.

Let us say that the conjugations $\rho$ and $\rho'$ are {\it equivalent} if $\rho'=h\rho h^{-1}$
for some automorphism~$h$. Clearly, it makes sense to distinguish conjugations only up to equivalence, since
equivalent conjugations are ``the same for all practical purposes''.

\begin{Definition} If $A$ is equipped with a conjugation $\rho$ then we call a star-product $*$ on $A$ {\it conjugation-invariant} if $\rho(a*b)=\rho(a)*\rho(b)$.
\end{Definition}

It is easy to see that conjugation-invariant star-products on $A$ correspond
to quantizations~${\mathbf A}$ of~$A$ with a conjugation (whose associated graded is the conjugation of~$A$) equipped with a~conjugation-invariant quantization map~$\phi$.

\begin{Definition} We say a conjugation-invariant {\sl short} star-product $*$ on~$A$ is {\it Hermitian} if $CT(a*b)=CT\big(b*\rho^2(a)\big)$ for $a,b\in A$.
\end{Definition}

This terminology is motivated by the following lemma.

\begin{Lemma}\label{herm} Let $*$ be a conjugation-invariant short star-product on~$A$. Then~$*$ is Hermitian if and only if the sesquilinear form $a,b\mapsto \langle a,\rho(b)\rangle=CT(a*\rho(b))$ is Hermitian.
\end{Lemma}

\begin{proof} We have
\begin{gather*}
\overline{CT\big(\rho^{-1}(b)*\rho(a)\big)}=CT\big(b*\rho^2(a)\big).
\end{gather*}
The Hermitian condition is that this equals $CT(a*b)$, as desired.
\end{proof}

\begin{Definition}We will say that a Hermitian short star-product~$*$ on~$A$ is {\it unitary} if the Hermitian form $\langle a,\rho(b)\rangle$ is positive definite (i.e., it is positive definite on~$A_d$ for each~$d$).
\end{Definition}

\begin{Lemma}\label{lh1} A conjugation-invariant short star-product $*$ is nondegenerate if and only if
the form $\langle a,\rho(b)\rangle$ is nondegenerate; in particular, a~unitary star-product is automatically nondegenerate.
\end{Lemma}

\begin{proof} Straightforward.
\end{proof}

Also, recall that if $*$ is short and nondegenerate then it corresponds to a nondegenerate $g$-twisted trace $T$ on some quantization ${\mathbf A}$ of~$A$, and vice versa. Moreover, if $A$ is equipped with a conjugation map $\rho$ then we have an antilinear map $\rho\colon {\mathbf A}\to {\mathbf A}$ (as ${\mathbf A}$ is identified with $A$ as a~filtered vector space).

\begin{Lemma}\label{lh2}\quad
\begin{enumerate}\itemsep=0pt
\item[$(i)$] The star-product $*$ is conjugation-invariant if and only if~$\rho$ is a conjugation on ${\mathbf A}$ $($i.e., an antilinear algebra automorphism$)$ which conjugates~$T$ $($hence commutes with~$g)$;
\item[$(ii)$] In this situation $*$ is Hermitian if and only if in addition $\rho^2=g$.
\end{enumerate}
\end{Lemma}

\begin{proof} Straightforward.
\end{proof}

Let us now restrict to the case $A=\mathcal{O}(X)$ where $X$ is a conical symplectic singularity.

\begin{Lemma}\label{lh3} A lift of a conjugation $\rho$ from $A$ to ${\mathbf A}$ is unique if exists.
\end{Lemma}

\begin{proof} Any two lifts differ by an automorphism $\gamma$ of~${\mathbf A}$ that commutes with~$s$ and acts trivially on~$A$, and any such automorphism is the identity since any filtration-preserving derivation of ${\mathbf A}$ is inner (this follows from Lemma~\ref{inn} by taking the associated graded).
\end{proof}

Moreover, if $A$ is equipped with a conjugation then the space $\mathfrak{P}$ parametrizing quantizations of~$A$ carries a real structure $\lambda\mapsto \overline{\lambda}$ induced by~$\rho$.

\begin{Lemma}\label{lh4} A quantization ${\mathbf A}_\lambda$ admits a lift of $\rho$ if and only if $\overline{\lambda}\in W\lambda$, i.e., $\overline{\lambda}=\lambda$ in $\mathfrak{P}/W$, where $W$ is the Namikawa Weyl group.
\end{Lemma}

\begin{proof} Straightforward.
\end{proof}

\subsection{Quaternionic structures}

As before, let $A$ be a graded Poisson algebra with Poisson bracket of degree $-2$.

\begin{Definition} We say that a~conjugation $\rho\colon A\to A$ is a {\it quaternionic structure}
if $\rho^2=s$.
\end{Definition}

If $\rho$ is a quaternionic structure then the operators $\rho$ and $\sigma:={\rm i}^d$ define an (${\mathbb R}$-linear) action of the quaternion group~$Q_8$ by algebra automorphisms of~$A$, which justifies the terminology. Also it is clear that if~$\rho$ is a quaternionic structure and $\rho'$ is equivalent to $\rho$ then $\rho'$ is also a~quaternionic structure.

\begin{Definition} Let ${\mathbf A}$ be an {\it even} quatization of~$A$. We say that a conjugation $\rho\colon {\mathbf A}\to {\mathbf A}$ is a {\it quaternionic structure} on~${\mathbf A}$ if $\rho^2=s$ and $\rho \circ \sigma=\sigma^{-1}\circ \rho$.
\end{Definition}

It is clear that any quaternionic structure on ${\mathbf A}$ gives rise to a quaternionic structure on $A$ (by taking the associated graded), and a quaternionic structure on $A$ uniquely lifts to a quantization ${\mathbf A}={\mathbf A}_\lambda$ in the case of conical symplectic singularities when $\overline{\lambda}=\lambda=-\lambda\in \mathfrak{P}/W$.

Now suppose that $A$ is a Poisson algebra with a quaternionic structure $\rho$ and a short conjugation-invariant star-product $*$.

\begin{Lemma} The star-product $*$ is Hermitian if and only if it is even.
\end{Lemma}

 \begin{proof} If $a,b\in A_d$ then
\begin{gather*}
CT(\rho(a)*b)=\overline{CT(\rho(\rho(a)*b))}=\overline{CT(s(a)*\rho(b))}=
(-1)^d\overline{CT(a*\rho(b))}.
\end{gather*}
This implies the statement.
\end{proof}

Thus, conjugation-invariant Hermitian star-products on a quaternionic Poisson algebra $A$ correspond to quaternionic quantizations ${\mathbf A}$ of $A$ together with a quantization map $\phi$ which commutes with $\sigma$ and $\rho$.

\subsection{The quarternionic structure on a HyperK\"ahler cone} \label{hypcoh}

If $A=\mathcal{O}(X)$ where $X$ is a hyperK\"ahler cone then $A$ has a canonical quaternionic structure. Indeed, $X$ has a natural real analytic ${\rm SU}(2)$-action (see \cite[Appendix~A and references therein]{BPR}), so we have degree-preserving operators $I$, $J$, $K$ on the space ${\mathcal O}(X_{{\mathbb R}})$ of real analytic polynomial functions on $X$ such that $IJ=JIs=K$ and $I^2=J^2=s$. The operator $I$ actually preserves the subspace of holomorphic polynomials, $A={\mathcal O}(X)$, and acts on $A_d$ by ${\rm i}^d$, i.e., $I=\sigma$. On the other hand, $J$ and $K$ map holomorphic polynomials to antiholomorphic ones. Thus we have an antilinear Poisson automorphism of $A$ given by $\rho(f)=\overline{Jf}$. Thus $\rho$ defines a quaternionic structure on~$A$, which lifts uniquely to every quantization ${\mathbf A}_\lambda$ of $A$ with $\overline{\lambda}=\lambda=-\lambda\in \mathfrak{P}/W$.

\begin{Example} The simplest example of a hyperK\"ahler cone is a finite dimensional right quaternionic vector space $V$ with a positive definite quaternionic Hermitian inner product. Then $V$ carries a Euclidean metric $(\,,\,)$ given by the real part of this inner product, and three real-linear operators $I$, $J$, $K$ corresponding to three complex structures on~$V$, which gives a hyperK\"ahler structure on~$V$. Let~$A$ be the algebra of polynomial functions on the complex vector space $(V,I)$. We have three real symplectic forms on~$V$,
\begin{gather*}
\omega_I(v,w)=(Iv,w),\qquad \omega_J(v,w)=(Jv,w),\qquad \omega_K(v,w)=(Kv,w).
\end{gather*}
The form $\omega=\omega_J+{\rm i}\omega_K$ is holomorphic for the complex structure $I$, so defines a Poisson bracket on $A$ of degree $-2$, whose symmetry group is ${\rm Sp}(V)$. The symmetry group of $((\,,\,),I,J,K)$ is a maximal compact subgroup ${\rm U}(V)$ of ${\rm Sp}(V)$, the quaternionic unitary group. As explained above, we have $\sigma(f)(v)=f({\rm i}v)$, $\rho(f)(v)=\overline{f(vJ)}$.

Now let $V$ be a complex vector space with a symplectic form $\omega$ and $G\subset {\rm Sp}(V)$ be a finite subgroup. There exists a positive definite quaternionic Hermitian structure on $V$ such that $G\subset {\rm U}(V)$, which is unique up to a positive scalar if $V$ is indecomposable as a symplectic representation of $G$. A choice of such a structure gives us a form $(\,,\,)$ and operators $I$, $J$, $K$ preserved by $G$, so
$X=V/G$ is also a hyperK\"ahler cone, and $A:=\mathcal{O}(X)$ is a quaternionic Poisson algebra.
\end{Example}

Let us now classify quaternionic structures on a hyperK\"ahler cone~$X$. As explained above, there is a canonical one, call it $\rho_0$. Note that $\operatorname{Ad}\rho_0$ is an involution on $\operatorname{Aut}({\mathbf A})$ which corresponds to its compact real form, denote it by $\gamma\mapsto \overline{\gamma}$. Now, any quaternionic structure $\rho$ is of the form \mbox{$\rho=\rho_0\gamma$} for some automorphism $\gamma\in \operatorname{Aut}({\mathbf A})$. The condition on~$\gamma$ is that $s=\rho^2=\rho_0\gamma\rho_0\gamma=s\overline \gamma\gamma$, i.e., $\overline\gamma\gamma=1$. Moreover, replacing~$\rho$ by an equivalent conjugation $h\rho h^{-1}$ corresponds to replacing~$\gamma$ with~$\overline{h}\gamma h^{-1}$. This implies that
equivalence classes of quaternionic structures correspond to the real forms of $\operatorname{Aut}({\mathbf A})$
given by the Galois cohomology classes in $H^1({\mathbb Z}/2,\operatorname{Aut}({\mathbf A}))$, where~${\mathbb Z}/2$ acts by $\gamma\mapsto \overline{\gamma}$; i.e., to real forms which are in the inner class of the compact form.\footnote{Note that the group $\operatorname{Aut}({\mathbf A})$ may be disconnected.}

\begin{Example} If $X$ is the nilpotent cone of a semisimple Lie algebra $\g$ (so $s=1$) then ${\mathbf A}={\mathbf A}_\lambda$ is a maximal quotient of the enveloping algebra ${\rm U}(\g)$ by the central character $\chi=\chi_\lambda$ of $\lambda$, and for conjugations to exist, this central character should be real. Then $\operatorname{Aut}({\mathbf A})$ is the subgroup $\operatorname{Aut}(\g)_\chi$ of $\operatorname{Aut}(\g)$ that preserves $\chi$. So the real forms we get are those which differ from the compact form by an involution $\tau\in \operatorname{Aut}(\g)_\chi$ (i.e., this form is defined by the equation $\overline{g}=\tau(g)$). In particular, if the automorphism group of the Dynkin diagram of $\g$ preserves $\chi$ (e.g., if this group is trivial) then conjugations up to equivalence correspond to real forms of $\g$. So for $\g=\mathfrak{sl}_2$ we have two conjugations up to equivalence~-- the compact one and the noncompact one.
\end{Example}

\subsection{Classification of Hermitian short star-products}

Let $X$ be a conical symplectic singularity equipped with a conjugation $\rho\colon A\to A$ of $A=\mathcal{O}(X)$ (for example, a hyperK\"ahler cone, cf.\ Section~\ref{hypcoh}). In this case, given $\lambda\in \mathfrak{P}$ with $\overline{\lambda}\in W\lambda$, we have the corresponding unique lift $\rho\colon {\mathbf A}_\lambda\to {\mathbf A}_\lambda$. Let $g=\rho^2$.

Recall that by Corollary~\ref{bun} nondegenerate short star-products for ${\mathbf A}_\lambda$ corresponding to $g$ are classified by nondegenerate traces $T\in (HH_0({\mathbf A}_\lambda,{\mathbf A}_\lambda g)^*)^s$ such that $T(1)=1$. Note that $g$ acts trivially on $(HH_0({\mathbf A}_\lambda,{\mathbf A}_\lambda g)^*)^s$, hence $\rho$ defines an antilinear involution $\rho_*$ (i.e., a real structure) on this vector space. Therefore, by Lemma~\ref{lh2} we have the following proposition.

\begin{Proposition} \label{hermi} With the above assumptions, Hermitian nondegenerate short star-products for ${\mathbf A}_\lambda$ correspond to real nondegenerate traces $T\in (HH_0({\mathbf A}_\lambda,{\mathbf A}_\lambda g)^*)^s$ such that $T(1)=1$, i.e., those invariant under $\rho_*$.
\end{Proposition}

In particular, nondegenerate short star-products corresponding to $g$ exist if and only if there exist Hermitian ones (and in this case a Weil generic short star-product is nondegenerate). In other words, the Hermitian property ``goes along for the ride".

\begin{Example} Consider again the case when $X$ is the quadratic cone in ${\mathbb C}^3$, i.e., the $\mathfrak{sl}_2$ case. Recall that a nondegenerate short star-product exists if and only if the central character $\chi=\chi_\lambda:=\lambda(\lambda+2)/2$ is not equal to $n(n+2)/2$ where $n\ge 0$ is an integer.
Thus, for each such $\chi\in {\mathbb R}$ and each conjugation $\rho$ we have a unique Hermitian
even nondegenerate short star-product (the one invariant under $\mathfrak{sl}_2$).
One can show (see, e.g., \cite{BPR}) that for the compact conjugation, this product is never unitary.
On the other case, a direct computation shows that for the non-compact conjugation, this product is unitary
if and only if $\chi<0$ (see~\cite[Section~5.1]{BPR}). This is related to Bargmann's classification of unitary representations of ${\rm SL}_2({\mathbb R})$ (the principal and complementary series, which exist in this range).

Thus, for the nilpotent cone of higher rank simple Lie algebras (as well as orbit closures and Slodowy slices) we also expect a relation of the question of unitarity of short star-products and the theory of unitary representations of real forms of the Lie group~$G$ corresponding to~$\g$. Hopefully this will be a first step in the general theory of unitary representations of quantizations of symplectic singularities, which would generalize the theory of unitary representations of real reductive groups. We plan to discuss this in more detail in the second part of this work.
\end{Example}

\section{Existence of nondegenerate short star-products}\label{section4}

\subsection[The case $g=s$]{The case $\boldsymbol{g=s}$}

Let us now discuss existence of nondegenerate short star-products. We first focus on the case $g=s$, which includes the case of even star-products. We have shown that such star-products giving a quantization ${\mathbf A}$ are parametrized by a subset of the space $(HH_0({\mathbf A},{\mathbf A}s)^s)^*$ determined by a nondegeneracy condition.

\begin{Conjecture}\label{exi} \quad
\begin{enumerate}\itemsep=0pt
\item[$(i)$] For Weil generic $\lambda$ a Weil generic element $T\in (HH_0({\mathbf A}_\lambda,{\mathbf A}_\lambda s)^s)^*$ is nondegenerate, so it defines a nondegenerate short star-product corresponding to $g=s$.
\item[$(ii)$] For Weil generic $\lambda$ such that $-\lambda\in W\lambda$ a Weil generic element $T\in (HH_0({\mathbf A}_\lambda,{\mathbf A}_\lambda s)^\sigma)^*$ is nondegenerate, so it defines
an even nondegenerate short star-product.
\end{enumerate}
\end{Conjecture}

This conjecture generalizes a conjecture from~\cite{BPR}.

Below we check Conjecture~\ref{exi} in a number of important special cases.

\subsubsection{Spherical symplectic reflection algebras}\label{ssra}

Let $V$ be a finite dimensional symplectic vector space with symplectic form $\omega$,
and $G\subset {\rm Sp}(V)$ a finite subgroup such that $\big({\wedge}^2V^*\big)^G={\mathbb C}\omega$.
Let $S$ be the set of symplectic reflections in~$G$ (i.e., elements $g$ such that $(1-g)|_V$ has rank $2$)
and $c\colon S\to {\mathbb C}$ a $G$-invariant function. The {\it symplectic reflection algebra} ${\mathbf{H}}_c={\mathbf{H}}_c(G,V)$ attached to this data is the quotient of ${\mathbb C}G\ltimes TV$ by the relations
\begin{gather*}
[v,w]=\omega(v,w)+\sum_{g\in S}c_g\omega((1-g)v,(1-g)w)g,
\end{gather*}
see \cite{EG}. Let ${\mathbf{e}}=|G|^{-1}\sum\limits_{g\in G}g$. The spherical subalgebra of ${\mathbf{H}}_c$ is the algebra ${\mathbf A}_c:={\mathbf{e}} {\mathbf{H}}_c{\mathbf{e}}$. It is shown in~\cite{EG} that ${\mathbf A}_c$ is a filtered deformation of $A:=(SV)^G$. Moreover, by the result of~\cite{L2}, this is the most general such deformation (as in this case $\mathfrak{P}\cong {\mathbb C}[S]^G$); here $\lambda$ is related to $c$ by a nonhomogeneous invertible linear transformation.

\begin{Theorem} Conjecture~{\rm \ref{exi}} holds for ${\mathbf A}_c$.
\end{Theorem}

\begin{proof} Recall that the automorphism $s$ is defined by $s|_V=-1$. Let us compute the space $HH_0({\mathbf A}_c,{\mathbf A}_cs)^s$ for Weil generic $c$. By the results of \cite{EG}, the algebra ${\mathbf A}_c$ is Morita equivalent to~${\mathbf{H}}_c$, so this space is isomorphic to $HH_0({\mathbf{H}}_c,{\mathbf{H}}_cs)^s$ (as this Morita equivalence is $s$-invariant).

Consider first the case when $s\in G$. In this case, $HH_0({\mathbf{H}}_c,{\mathbf{H}}_cs)^s=HH_0({\mathbf{H}}_c,{\mathbf{H}}_c)$. This group was computed in \cite{EG} and is isomorphic to ${\mathbb C}[\Sigma_+]^G$, where $\Sigma_+$ is the set of elements of $G$ such that $(1-g)|_V$ is invertible. Note also that the same result holds for $c=0$.

Now consider the case $s\notin G$. In this case a similar argument shows that
$HH_0({\mathbf{H}}_c,{\mathbf{H}}_cs)^s$ is isomorphic to ${\mathbb C}[\Sigma_-]^G$, where $\Sigma_-$ is the set of elements of $G$ such that $(1+g)|_V$ is invertible, and again this result holds for $c=0$. Note that this holds regardless of whether~$s$ belongs to~$G$, since when $s\in G$ then multiplication by $s$ defines a bijection
$\Sigma_+\to \Sigma_-$. Note also that this space is always nonzero since $1\in \Sigma_-$.

Finally, ${\mathbf A}_0={\rm Weyl}(V)^G$, and this algebra admits a nondegenerate $s$-twisted trace corresponding to the Moyal--Weyl product. By a deformation argument, this implies that for Weil generic $c\ne 0$, a Weil generic
element of ${\mathbb C}[\Sigma_-]^G$ gives rise to a nondegenerate $s$-twisted trace, as desired.

Finally, in the even case we have to show that we have a nondegenerate $\sigma$-invariant trace.
It is easy to see that for $c=0$ the map $\sigma$ acts on ${\mathbb C}[\Sigma_-]^G$ by $\sigma(g)=g^{-1}$.
This shows that even nondegenerate star-products exist for Weil generic $c$ and are parametrized by Weil generic $G$-invariant functions on $\Sigma_-$ stable under $g\mapsto g^{-1}$, modulo scaling.
\end{proof}

\begin{Example} Suppose $\dim V=2$. This corresponds to the case of Kleinian singularities,
and the possible groups $G$ are in bijection with simply-laced finite Dynkin diagrams
via the McKay's correspondence. Namely, $G={\mathbb Z}/m$ for $A_{m-1}$, $G$ is the double cover
of the dihedral group~$I_{m-2}$ for~$D_m$, of the tetrahedral group ${\mathbb A}_4$ for~$E_6$, the cube group~$S_4$ for $E_7$ and the icosahedral group~${\mathbb A}_5$ for~$E_8$. The Namikawa Weyl group coincides with the usual Weyl group of the corresponding root system. Looking at the conjugacy classes of elements not equal to~$-1$ modulo $g\mapsto g^{-1}$, we find the following facts about even nondegenerate short star-products.
\begin{enumerate}\itemsep=0pt
\item The number $m_e$ of parameters for even quantizations equals the number of orbits of the dualization involution (for Lie algebra representations) on the finite Dynkin diagram. So~$m_e$ equals $\big\lfloor\frac{m+1}{2}\big\rfloor$ for~$A_m$, $m-1$ for $D_m$ with $m$ odd, $4$ for $E_6$, and $m$ in all other cases.
\item The number of parameters for even nondegenerate short star-products for each quantization is $m_e-1$, except the case of~$A_m$ with even~$m$, when it equals~$m_e$.
\end{enumerate}
\end{Example}

\subsubsection{Maximal quotients of the enveloping algebra of a simple Lie algebra} \label{maxquot}

Let $\g$ be a finite dimensional simple Lie algebra with Cartan subalegbra $\h$, ${\rm U}(\g)$ its enveloping algebra, $Z(\g)$ the center of ${\rm U}(\g)$, $\lambda\in \h^*$ a complex weight for $\g$, and $\chi\colon Z(\g)\to {\mathbb C}$ the central character of the Verma module with highest weight $\lambda$. Let ${\mathbf A}_\lambda:=U_\chi(\g)={\rm U}(\g)/(z-\chi(z), z\in Z(\g))$ be the corresponding maximal primitive quotient. Then ${\mathbf A}_\lambda$ is naturally a filtered quantization of $A=\mathcal{O}(X)$, where~$X$ is the nilpotent cone of $\g$ (where elements of~$\g$ have degree~$2$). In this case, since all the degrees are even, we have $s=1$. Moreover, if $\chi$ is self-dual, the antipode of~${\rm U}(\g)$ defines the antiautomorphism~$\sigma$ of ${\mathbf A}_\lambda$, so that ${\mathbf A}_\lambda$ is an even quantization of~$A$.

Let us first classify nondegenerate untwisted traces. It is well known that $HH_0({\mathbf A}_\lambda,{\mathbf A}_\lambda)={\mathbb C}$. Namely, if $\chi$ is the central character of a finite dimensional representation~$V$ then the unique trace~$T$ normalized by the condition $T(1)=1$ is given by the normalized character of~$V$:
\begin{gather*}
T(a)=\frac{\operatorname{Tr}(a|_V)}{\dim V}.
\end{gather*}
This trace is defined by a rational function of~$\chi$, so for general $\chi$ it can be obtained by interpolation.

\begin{Proposition}\label{nondeLie} $T$ is nondegenerate for Weil generic~$\lambda$ $($or~$\chi)$.
\end{Proposition}

\begin{proof} It suffices to show that for each $i$, $T$ is nondegenerate on $F_i{\mathbf A}_\lambda$ for Weil generic~$\lambda$. To this end, it suffices to show that this is so for at least one $\lambda$. To do so, take $\lambda$ dominant integral and consider the irreducible finite dimensional representation $V=V_\lambda$ of $\g$. We claim that for each~$i$ one can pick $\lambda$ so that $F_i{\mathbf A}_\lambda$ has trivial intersection with the annihilator $\operatorname{Ann}(V)$ of $V$ in~${\mathbf A}_\lambda$. Indeed, recall that as a $\g$-module, ${\mathbf A}_\lambda=\oplus_\mu V_\mu\otimes V_\mu^*[0]$, where $\mu$ runs over dominant integral weights. On the other hand, $\operatorname{End}(V)=\oplus_\mu V_\mu\otimes \operatorname{Hom}(V_\mu\otimes V,V)$. So if $\operatorname{Hom}(V_\mu\otimes V,V)\cong V_\mu^*[0]$ for all $V_\mu$ occurring in $F_i{\mathbf A}_\lambda$, then $F_i{\mathbf A}_\lambda\cap \operatorname{Ann}(V)=0$. Clearly, this happens if all the coordinates of $\lambda$ are large enough.

Now note that $V$ is a unitary representation, so $\operatorname{End}(V)$ has an antilinear antiautomorphism
$B\mapsto B^\dagger$, such that $\operatorname{Tr}_V\big(B^\dagger B\big)>0$ for $B\ne 0$. Let $\pi_V\colon {\mathbf A}_\lambda\to \operatorname{End}(V)$ be the representation map. We have $\pi_V({\mathbf a})^\dagger=\pi_V\big({\mathbf a}^\dagger\big)$, where ${\mathbf a}\mapsto {\mathbf a}^\dagger$ is the compact $*$-structure on ${\mathbf A}_\lambda$. For ${\mathbf a}\in {\mathbf A}_\lambda$, we have
\begin{gather*}
\big({\mathbf a},{\mathbf a}^\dagger\big)_T=\frac{\operatorname{Tr}_V\big(\pi_V({\mathbf a})\pi_V({\mathbf a})^\dagger\big)}{\dim V}.
\end{gather*}
If ${\mathbf a}\in F_i{\mathbf A}_\lambda$ then $\pi_V({\mathbf a})\ne 0$, so we get that $\big({\mathbf a},{\mathbf a}^\dagger\big)_T>0$. Thus $(\,,\,)_T$ is nondegenerate on $F_i{\mathbf A}_\lambda$, as desired.
\end{proof}

\begin{Remark}\label{rich} \looseness=1 More generally, let $y\in \h\subset \g$ be a semisimple element, $L=Z_y$ its centralizer (a~Levi subgroup of the corresponding group $G$) and $P$ a parabolic in $G$ whose Levi subgroup is~$L$. Let~$U$ be the unipotent radical of~$P$ and $u\in \operatorname{Lie}U\subset \g$ a Richardson element (i.e., one whose $P$-orbit is dense in $\operatorname{Lie}U$), which exists by a classical theorem of Richardson. We view~$u$ as an element of~$\g^*$ and consider the coadjoint orbit $\mathcal{O}_u$ of $u$, called the Richardson orbit corresponding to~$y$ (e.g., for $\g=\mathfrak{sl}_n$ any nilpotent orbit arises in this way). Let~$X$ be the closure of $\mathcal{O}_u$. Let $J\subset I$ be the set of $j$ such that $\alpha_j(y)=0$ (e.g., if $X$ is the nilpotent cone then~$y$ is regular and $J=\varnothing$ and if $X=0$ then $y=0$ and $J=I$). In this case $\mathcal O(X)$ admits a~family of quantizations ${\mathbf A}_\lambda$ parametrized by weights~$\lambda$ such that $\lambda(h_i)=0$ for $i\in J$; namely, ${\mathbf A}_\lambda$ is the image of ${\rm U}(\g)$ in the ${\mathbb C}$-endomorphism algebra of the parabolic Verma module $M_P(\lambda)$ with highest weight $\lambda$ (generated by a vector $v$ with defining relations
$e_iv=0$, $h_iv=\lambda(h_i)v$ for all $i$ and $f_iv=0$ for $i\in J$). In this case, we still have $s=1$ and $HH_0({\mathbf A}_\lambda,{\mathbf A}_\lambda)={\mathbb C}$, so we have a unique trace up to scaling. A similar positivity argument to the above, using finite dimensional quotients of $M_P(\lambda)$ (which occur when $\lambda$ is dominant integral), shows that this trace is nondegenerate for Weil generic $\lambda$, hence gives rise to a nondegenerate even short star-product.
\end{Remark}

\subsection{The general case}

Finally, let us describe the construction of nondegenerate twisted traces for generic $g\in \operatorname{Aut}({\mathbf A})^0$, in the case of hyperK\"ahler cones. We plan to give more details and also treat the case when $g\notin \operatorname{Aut}({\mathbf A})^0$ in a subsequent paper.

\subsubsection{Characters of Verma modules}

We may assume that $g$ is in a maximal torus ${\mathbf T}\subset \operatorname{Aut}({\mathbf A})$. Let us adapt the construction from Sections~\ref{repth1},~\ref{repth} which uses representation theory (category $\mathcal{O}$). Namely, let $\nu\colon {\mathbb C}^\times\to {\mathbf T}$ be a~1-parameter subgroup.
By Lemma~\ref{inn}, the derivation defined by~$\nu$ is inner and given by some element ${\mathbf{h}}\in {\mathbf A}$, such that the operator $[{\mathbf{h}},?]$ on ${\mathbf A}$ is diagonalizable with integer spectrum. For $m\in {\mathbb Z}$ denote by ${\mathbf A}(m)$ the space of ${\mathbf a}\in {\mathbf A}$ such that $[{\mathbf{h}},{\mathbf a}]=m{\mathbf a}$; we have ${\mathbf A}=\oplus_{m\in {\mathbb Z}}{\mathbf A}(m)$.

Following \cite[Section~4]{L4} and \cite{BLPW} (see also~\cite{G}), in this situation one can define the category $\mathcal{O}=\mathcal{O}(\nu,\lambda)$ of representations of ${\mathbf A}={\mathbf A}_\lambda$. Namely, we define
\begin{gather*}
{\mathbf A}^{>0}:=\oplus_{m>0}{\mathbf A}(m), \qquad {\mathbf A}^{<0}:=\oplus_{m<0}{\mathbf A}(m),\\
{\mathbf A}^{\ge 0}:=\oplus_{m\ge 0}{\mathbf A}(m), \qquad {\mathbf A}^{\le 0}:=\oplus_{m\le 0}{\mathbf A}(m).
\end{gather*}
Also we define the algebra
\begin{gather*}
C_\nu({\mathbf A})={\mathbf A}(0)/\sum_{m>0}{\mathbf A}(-m){\mathbf A}(m),
\end{gather*}
called the {\it Cartan subquotient} of ${\mathbf A}$. We define the category $\mathcal O(\nu,\lambda)$ to be the category of finitely generated ${\mathbf A}$-modules $M$ on which ${\mathbf A}^{>0}$ acts locally nilpotently (i.e., for every vector $v\in M$ the space ${\mathbf A}^{>0}v$ is finite dimensional and has a filtration strictly preserved by ${\mathbf A}^{>0}$).

Let us assume that the only fixed point of ${\mathbf T}$ on $X$ is the origin. In this case we can (and will) choose $\nu$ to be so generic
that $X^{\nu({\mathbb C}^\times)}=X^{{\mathbf T}}$ as schemes. Then the algebra $C_\nu({\mathbf A})$ is finite dimensional and carries a trivial action of ${\mathbf T}$. In this case, it is shown in \cite[Lemma~4.1]{L4} that the category $\mathcal{O}$ is artinian with finitely many simple objects and enough projectives and each object has finite dimensional generalized eigenspaces of ${\mathbf{h}}$. Also $\mathcal{O}$ has an important class of objects called Verma modules. Namely, given an (irreducible) $C_\nu({\mathbf A})$-module $\tau$, the Verma module~$\Delta(\tau)$ with lowest weight $\tau$ is, by definition, generated by a copy of $\tau$ (regarded as an ${\mathbf A}(0)$-module) with defining relations $av=0$ for $a\in {\mathbf A}^{<0}$, $v\in \tau$.

Let us equip $\tau$ with the trivial action of ${\mathbf T}$. This defines an action of ${\mathbf T}$ on $\Delta(\tau)$ compatible with the ${\mathbf A}$-action. Define the {\it reduced character} of $\Delta(\tau)$ by the formula
\begin{gather*}
{\rm Ch}_\tau({\mathbf a},g,t)=\operatorname{Tr}_{\Delta(\tau)}({\mathbf a} g\nu(t)):=\sum_{m\ge 0}\operatorname{Tr}_{\Delta(\tau)[m]}({\mathbf a} g)t^m\in {\mathbb C}[[t]],
\end{gather*}
where ${\mathbf a}\in {\mathbf A}(0)$ and $g\in {\mathbf T}$. By looking at the constant terms of these series, it is clear that these characters for different irreducible representations $\tau$ are linearly independent over ${\mathbb C}((t))$ (for each $g\in {\mathbf T}$).

It is clear that ${\rm Ch}_\tau(?,g,t)$ is a $g\nu(t)$-twisted trace on the algebra ${\mathbf A}^t:={\mathbf A}\otimes {\mathbb C}((t))$. Thus we obtain

\begin{Proposition} We have a ${\mathbb C}((t))$-linear injection
\begin{gather*}
{\rm Ch}\colon \ K_0(C_\nu({\mathbf A}))((t))\hookrightarrow HH_0\big({\mathbf A}^t,{\mathbf A}^t g\nu(t)\big)^*.
\end{gather*}
\end{Proposition}

\subsubsection{Differential equations for reduced characters}

Let ${\mathbf b}_i$, $i=1,\dots,r$ be a collection of elements of ${\mathbf A}(0)$ such that
any element ${\mathbf a}\in F_k{\mathbf A}(0)$ can be written as
\begin{gather*}
{\mathbf a}=\sum_i \lambda_i{\mathbf b}_i+\sum_j {\mathbf c}_j^-{\mathbf c}_j^+, \qquad \lambda_i\in {\mathbb C},
\end{gather*}
where ${\mathbf c}_j^\pm\in {\mathbf A}(\pm m_j)$ for some $m_j>0$, so that $\deg({\mathbf c}_j^+)+\deg({\mathbf c}_j^-)\le k$. Such a collection exists since the only fixed point of $\nu({\mathbb C}^\times)$ on $X$ is the origin. Moreover, we may assume that ${\mathbf c}_j^\pm$ have weights
$\pm \mu_j$ under ${\mathbf T}$.

Let $T$ be a $g\nu(t)$-twisted trace on ${\mathbf A}^t$. Then we get
\begin{gather*}
T({\mathbf a})=\sum_i \lambda_iT({\mathbf b}_i)+\sum_j T({\mathbf c}_j^-{\mathbf c}_j^+).
\end{gather*}
On the other hand,
\begin{gather*}
T({\mathbf c}_j^-{\mathbf c}_j^+)=T([{\mathbf c}_j^-,{\mathbf c}_j^+])+T({\mathbf c}_j^+{\mathbf c}_j^-)=
T([{\mathbf c}_j^-,{\mathbf c}_j^+])+\mu_j(g\nu(t))T({\mathbf c}_j^-{\mathbf c}_j^+),
\end{gather*}
which implies that
\begin{gather*}
T({\mathbf c}_j^-{\mathbf c}_j^+)=\frac{T([{\mathbf c}_j^-,{\mathbf c}_j^+])}{1-\mu_j(g\nu(t))}.
\end{gather*}
Thus we get
\begin{gather*}
T({\mathbf a})=\sum_i \lambda_iT({\mathbf b}_i)+\sum_j \frac{T([{\mathbf c}_j^-,{\mathbf c}_j^+])}{1-\mu_j(g\nu(t))}.
\end{gather*}
But the element $[{\mathbf c}_j^-,{\mathbf c}_j^+]$ has filtration degree $\le k-2$. Thus, continuing this process,
we obtain the following lemma.

Let $R\subset {\mathbb C}({\mathbf T})$ be the subalgebra generated by elements of the form $\frac{1}{1-\mu}$, where $\mu$ is a~character of ${\mathbf T}$ occurring in ${\mathbf A}(m)$ for some $m>0$.

\begin{Lemma}\label{normfor}
For any ${\mathbf a}\in {\mathbf A}$ there exist $p_i\in R$, $i=1,\dots,r$ $($independent of~$T)$ such that
\begin{gather*}
T({\mathbf a})=\sum_i p_i(g\nu(t))T({\mathbf b}_i).
\end{gather*}
\end{Lemma}

\begin{Corollary} \label{normfor1} Let $h_\tau$ be the eigenvalue of ${\mathbf{h}}$ on $\tau\subset \Delta(\tau)$. There exist $p_{ij}\in R$ independent of~$\tau$ such that the power series ${\rm Ch}_\tau({\mathbf b}_i,g,t)$ satisfy the system of linear differential equations
\begin{gather*}
t\frac{{\rm d}}{{\rm d}t}{\rm Ch}_\tau({\mathbf b}_i,g,t)=\sum_{j=1}^r \big(p_{ij}(g\nu(t))-h_\tau \delta_{ij}\big){\rm Ch}_\tau({\mathbf b}_j,g,t).
\end{gather*}
In particular, these power series converge in some neighborhood of $0$ and extend to $($a~priori, possibly multivalued$)$ analytic functions outside finitely many points $t={\rm e}^{2\pi ij/m}\mu(g)^{1/m}$ where $\mu$ is a ${\mathbf T}$-weight occurring in ${\mathbf A}(m)$, $m>0$.
\end{Corollary}

\begin{proof} We have
\begin{gather*}
t\frac{{\rm d}}{{\rm d}t}{\rm Ch}_\tau({\mathbf b}_i,g,t)={\rm Ch}_\tau(({\mathbf{h}}-h_\tau){\mathbf b}_i,g,t).
\end{gather*}
By Lemma~\ref{normfor}, there exist $p_{ij}\in R$ such that
\begin{gather*}
{\rm Ch}_\tau({\mathbf{h}}{\mathbf b}_i,g,t)=\sum_j p_{ij}(g\nu(t)) {\rm Ch}_\tau({\mathbf b}_j,g,t).
\end{gather*}
This implies the statement.
\end{proof}

In particular, for Weil generic $g$ the reduced character ${\rm Ch}_\tau({\mathbf b}_i,g,t)$ can be evaluated at $t=1$, giving a~$g$-twisted trace on ${\mathbf A}$. Thus we obtain an injective map
\begin{gather*}
{\rm Ch}|_{t=1}\colon \ K_0(C_\nu({\mathbf A}))\hookrightarrow HH_0({\mathbf A},{\mathbf A} g)^*.
\end{gather*}
for Weil generic $g\in {\mathbf T}$. We will see that this map is often an isomorphism, in which case we get a construction of all $g$-twisted traces via representation theory.

\subsubsection{Rationality of reduced characters}

\begin{Theorem}\label{ratfun} The series ${\rm Ch}_\tau({\mathbf a},g,t)$ converges to a rational function of $g\nu(t)$ with numerator in ${\mathbb C}[{\mathbf T}]$ and denominator a finite product of binomials of the form $1-\mu$, where $\mu$ is a character of ${\mathbf T}$ occurring in ${\mathbf A}(m)$ for some $m>0$.
\end{Theorem}

\begin{proof} Recall first the classical theorem of Carlson and Polya \cite{C,P}.

\begin{Theorem}\label{t1} Let $f(z)=\sum\limits_{n=0}^\infty c_nz^n\in {\mathbb Z}[[z]]$. If this series converges for $|z|<1$ and analytically continues to any larger open set then $f$ is rational.
\end{Theorem}

Theorem~\ref{t1} immediately implies

\begin{Corollary}\label{c1} Let $f(z)=\sum\limits_{n=0}^\infty c_nz^n\in {\mathbb Z}[[z]]^m$. Suppose that
$f$ satisfies a linear ODE of the form
\begin{gather*}
f'(z)=A(z)f(z),
\end{gather*}
where $A(z)\in \operatorname{Mat}_m({\mathbb Q}(z))$ is regular for $|z|<1$. Then $f$ is rational.
\end{Corollary}

Let us now use Corollary~\ref{c1} to prove the rationality of ${\rm Ch}_\tau({\mathbf a},g,t)$ for $g$ of finite order~$r$. Clearly, it is enough to check the rationality of the series
\begin{gather*}
f_{j,\tau}({\mathbf a},g,t):=\sum_{l=0}^{r-1}{\rm e}^{2\pi ijl}{\rm Ch}_\tau\big({\mathbf a},g^l,t\big)
\end{gather*}
for all $j\in [0,r-1]$.

By standard abstract nonsense, it suffices to consider the case when the ground field is $\overline{{\mathbb Q}}$. Moreover, since ${\mathbf A}$ is finitely generated, we may assume that the ground field is some (Galois) number field $K$.

Furthermore, restricting scalars, we may (and will) assume that $K={\mathbb Q}$. Indeed, let ${\mathbf A}_{{\mathbb Q}}$
be the algebra ${\mathbf A}$ regarded as a ${\mathbb Q}$-algebra. Let $\Gamma={\rm Gal}(K/{\mathbb Q})$.
It is clear that
\begin{gather}\label{tenspr}
{\mathbf A}_{{\mathbb Q}}\otimes K=\oplus_{\gamma\in \Gamma}{\mathbf A}^\gamma,
\end{gather}
where ${\mathbf A}^\gamma$
is ${\mathbf A}$ twisted by $\gamma$. Let $f_{j,\tau}^{{\mathbb Q}}$ be the series $f_{j,\tau}$ computed for the algebra ${\mathbf A}_{{\mathbb Q}}$, i.e., over~${\mathbb Q}$. By~\eqref{tenspr}, we have
\begin{gather*}
f_{j,\tau}^{{\mathbb Q}}({\mathbf a},g,t)=\sum_{\gamma\in \Gamma}\gamma(f_{j,\tau}({\mathbf a},g,t))=\operatorname{Tr}_{K/{\mathbb Q}}(f_{j,\tau}({\mathbf a},g,t)).
\end{gather*}
Now let $c_i$ be a basis of $K$ over ${\mathbb Q}$ and $c_i^*$ the dual basis with respect to the form
$\operatorname{Tr}_{K/{\mathbb Q}}(xy)$. Then
\begin{gather*}
f_{j,\tau}({\mathbf a},g,t)=\sum_i c_i^*\operatorname{Tr}_{K/{\mathbb Q}}\big(c_if_{j,\tau}({\mathbf a},g,t)\big)=\sum_i c_i^*f_{j,\tau}^{{\mathbb Q}}(c_i{\mathbf a},g,t).
\end{gather*}
Thus, it suffices to prove the statement for ${\mathbf A}_{{\mathbb Q}}$, as claimed.

Using Lemma \ref{normfor}, Corollary~\ref{normfor1}, and Corollary~\ref{c1}, we see that rationality of $f_{j,\tau}({\mathbf a},g,t)$ follows from the following proposition, which shows that an integer multiple of $f_{j,\tau}({\mathbf a},g,t)$ has integer coefficients.

\begin{Proposition}\label{pr1} There exists a basis~$B$ of $\Delta(\tau)$ compatible with the grading such that for each ${\mathbf a}\in {\mathbf A}(0)$ there exists an integer $N_{{\mathbf a}}\ge 1$ for which ${\mathbf a} B\subset \frac{1}{N_{{\mathbf a}}}{\mathbb Z} B$.
\end{Proposition}

\begin{proof} We have $\Delta(\tau)={\mathbf A}^{\le 0}\tau$. Recall that the algebra ${\mathbf A}^{\le 0}$ is finitely generated. Let $x_1,\dots,x_r$ be homogeneous generators of ${\mathbf A}^{\le 0}$, where $x_i$ have negative degree for~\mbox{$i\le s$} and degree~$0$ for~\mbox{$i>s$}. Then $x_{s+1},\dots,x_r$ generate ${\mathbf A}(0)$. So it suffices to show that $N_{{\mathbf a}}$ exists for ${\mathbf a}=x_i$, $i>s$, i.e., that $x_{s+1},\dots,x_r$ can be renormalized so that there is a basis $B$ such that $x_iB\subset {\mathbb Z} B$ for~\mbox{$i>s$}.

Let $d_i$ be the filtration degrees of $x_i$. We have commutation relations
\begin{gather*}
[x_i,x_j]=P_{ij}(x_1,\dots,x_r),
\end{gather*}
where $P_{ij}$ is a noncommutative polynomial of degree $\le d_i+d_j-2$. Multiplying
$x_i$ by $N^{d_i}$ for a suitable integer $N\ge 1$, we can make sure that the coefficients of $P_{ij}$
are in ${\mathbb Z}$. Then any monomial in $x_i$ can be written as a ${\mathbb Z}$-linear combination of ordered monomials in which the index is nondecreasing from left to right.

Let $v_1,\dots,v_n$ be a basis of $\tau$. The elements $x_i$, $i>s$ act by some $n$ by $n$ matrices in this basis. By adjusting the integer $N$ introduced above, we may make sure that these matrices have entries in ${\mathbb Z}$.

Now let $\tau_{\rm int}=\oplus_i {\mathbb Z}v_i\subset \tau$, and $\Delta(\tau)_{\rm int}={\mathbb Z}\langle x_1,\dots,x_s\rangle \tau_{\rm int}$. Since homogeneous subspaces of~$\Delta(\tau)$ are finite dimensional in each degree, this group is finitely generated in each degree, hence free, and $\Delta(\tau)_{\rm int}\otimes_{{\mathbb Z}}{\mathbb Q}=\Delta(\tau)$.

We claim that $\Delta(\tau)_{\rm int}$ is invariant under $x_i$ for any $i>s$. Indeed, it suffices to show that
$x_ix_1^{m_1}\cdots x_s^{m_s}v_j\in \Delta(\tau)_{\rm int}$. As explained above, we can write
$x_ix_1^{m_1}\cdots x_s^{m_s}$ as a ${\mathbb Z}$-linear combination of ordered monomials,
so it suffices to show that $x_1^{n_1}\cdots x_r^{n_r}v_j\in \Delta(\tau)_{\rm int}$. But this follows
from the definition, since $x_{s+1}^{n_{s+1}}\cdots x_r^{n_r}v_j$ is an integer linear combination of~$v_k$.

Thus, any homogeneous basis $B$ of $\Delta(\tau)_{\rm int}$ has the required property $x_iB\subset {\mathbb Z} B$ for $i>s$.
\end{proof}

Thus, Theorem~\ref{ratfun} for $g$ of finite order is proved. Let us now prove Theorem~\ref{ratfun} in ge\-neral. Lemma~\ref{normfor} and Corollary~\ref{normfor1} imply that the orders of the poles of the rational function ${\rm Ch}_\tau({\mathbf a},g,t)$ for~$g$ of finite order are uniformly bounded on a Zariski dense set $Z\subset {\mathbf T}$ of such~$g$. Thus there exists a finite product $\Pi$ of binomials of the form $1-\mu$ such that $\Pi(g\nu(t)){\rm Ch}_\tau({\mathbf a},g,t)$ is a polynomial in $t$ of a fixed degree for any $g\in Z$. But this product belongs to ${\mathbb C}[{\mathbf T}][[t]]$. This implies that $\Pi(g\nu(t)){\rm Ch}_\tau({\mathbf a},g,t)\in {\mathbb C}[{\mathbf T}][t]$. Since this is a function
of $g\nu(t)$, it has the form $P(g\nu(t))$ for some $P\in {\mathbb C}[T]$, as claimed.
\end{proof}

\begin{Remark} \looseness=-1 For ${\mathbf a}=1$ Theorem \ref{ratfun} follows from the Hilbert syzygies theorem. Indeed, the commutative algebra $\operatorname{gr}({\mathbf A}^{\ge 0})$ is finitely generated by the argument in \cite[proof of Lemma~3.1.2]{GL}.
\end{Remark}

Theorem \ref{ratfun} implies that it makes sense to set the auxiliary variable $t$ to $1$ and define the reduced character
\begin{gather*}
{\rm Ch}_\tau({\mathbf a},g):=\operatorname{Tr}_{\Delta(\tau)}({\mathbf a}g)=\frac{P(g)}{\prod\limits_{j=1}^r(1-\mu_j(g))},
\end{gather*}
where $P\in {\mathbb C}[{\mathbf T}]$, which is a rational function on ${\mathbf T}$.

Moreover, let ${\mathbf{h}}_1,\dots,{\mathbf{h}}_k$ be a basis of $\operatorname{Lie}{\mathbf T}\subset {\mathbf A}$, such that a general element of ${\mathbf T}$ can be written as $g=\prod\limits_{i=1}^k t_i^{{\mathbf{h}}_i}$. Let $h_{i,\tau}$ be the eigenvalue of ${\mathbf{h}}_i$ on $\tau\subset \Delta(\tau)$. Then
\begin{gather*}
{\rm Ch}_\tau({\mathbf a},g)=\operatorname{Tr}_{\Delta(\tau)}\left({\mathbf a} \prod_{i=1}^k t_i^{{\mathbf{h}}_i-h_{i,\tau}}\right).
\end{gather*}
This shows that a more natural definition of character is
\begin{gather*}
\widetilde{\rm Ch}_\tau({\mathbf a},g):=\prod_{i=1}^k t_i^{h_{i,\tau}}\cdot {\rm Ch}_\tau({\mathbf a},g)=\operatorname{Tr}_{\Delta(\tau)}\left({\mathbf a} \prod_{i=1}^k t_i^{{\mathbf{h}}_i}\right),
\end{gather*}
even though this is now in general a multivalued, no longer rational, function on~${\mathbf T}$. We call this function the {\it non-reduced, or ordinary character} of $\Delta(\tau)$ (as it agrees with the usual notions of character for representations of semisimple Lie algebras and rational Cherednik algebras). Similarly to Corollary~\ref{normfor1}, one can show that the functions $\widetilde{\rm Ch}_\tau({\mathbf b}_i,g)$ satisfy a holonomic system of linear differential equations:{\samepage
\begin{gather*}
\partial_l\widetilde{\rm Ch}_\tau({\mathbf b}_i,g)=\sum_{j=1}^r p_{ijl}(g)\widetilde{\rm Ch}_\tau({\mathbf b}_j,g),\qquad l=1,\dots,k,
\end{gather*}
where $\partial_l$ is the derivative along ${\mathbf{h}}_l$ and $p_{ijl}\in R$ are independent of $\tau$.}

Furthermore, such a character can be defined for any object $M\in \mathcal O$:
\begin{gather*}
\widetilde{\rm Ch}(M)({\mathbf a},g)=\operatorname{Tr}_M\left(a\prod_{i=1}^k t_i^{{\mathbf{h}}_i}\right).
\end{gather*}
If $M=\sum_\tau m_\tau \Delta(\tau)$ in the Grothendieck group of~$\mathcal O$ (a decomposition which always exists and is unique) then $\widetilde{\rm Ch}(M)=\sum_\tau m_\tau \widetilde{\rm Ch}_\tau$.

\subsubsection{Examples}

\begin{Example}\label{exa1} Let $G$ be a finite group, $\h$ an irreducible finite dimensional representation of~$G$, $V=\h\oplus \h^*$ and $X=V/G$ (this is a special case of the setting of Section~\ref{ssra}, when $V=\h\oplus \h^*$). In this case the symplectic reflection algebra ${\mathbf A}_c$ quantizing $\mathcal O(X)$ is called the {\it spherical rational Cherednik algebra}. Let $\nu(t)$ act by $t^{-1}$ on $\h$ and by $t$ on $\h^*$. Then the only fixed point of $\nu(t)$ is $0\in X$, so the above analysis applies. Consider a spherical parameter~$c$, i.e., such that ${\mathbf{H}}_c{\mathbf{e}}{\mathbf{H}}_c={\mathbf{H}}_c$ (such parameters form a nonempty Zariski open set). Then the algebra $C_\nu({\mathbf A})$ is Morita equivalent to ${\mathbb C}G$, so its irreducible representations $\tau$ correspond to irreducible representations $\overline\tau$ of ${\mathbb C}G$, with $\Delta(\tau)={\mathbf{e}} M(\overline\tau)$, where $M(\overline\tau)$ is the Verma module over ${\mathbf{H}}_c$ with highest weight~$\tau$. This gives a direct proof of Theorem~\ref{ratfun} (rationality of reduced characters) for the algebra ${\mathbf A}_c$ (since a similar statement is easily proved for ${\mathbf{H}}_c$ using its triangular decomposition). Also, in this case an argument similar to the proof of \cite[Theorem~1.8]{EG} shows that for generic~$g$ the dimension of $HH_0({\mathbf A},{\mathbf A} g)$ equals the number of conjugacy classes in $G$, which implies that the map ${\rm Ch}|_{t=1}$ is an isomorphism. Thus, all $g$-twisted traces in this case are obtained from category $\mathcal{O}$. Also, by deforming from $c=0$ we see that these traces are Weil generically nondegenerate, i.e., give rise to nondegenerate short star-products. Thus such star-products for fixed~$c$ and~$g$ depend on $r-1$ parameters, where~$r$ is the number of conjugacy classes in~$G$.
\end{Example}

\begin{Example}\label{exa2} Let $X$ be the nilpotent cone of a simple Lie algebra~$\g$, and keep the notation of Section~\ref{maxquot}. Let $\chi$ be a regular central character.
Let $\nu(t)$ be defined by the condition that $\alpha(\nu(t))=t$ for all positive simple roots $\alpha$.
In this case $C_\nu({\mathbf A})={\rm Fun}(W)$,
the algebra of complex-valued functions on the Weyl group~$W$,  and the category $\mathcal{O}$ is the usual
BGG category $\mathcal{O}_\chi$ for $\g$.

Moreover, in this case we can take ${\mathbf b}_i$ to be polynomials on
$\h$, so ${\rm Ch}_\tau({\mathbf b}_i,g,t)$ are obtained by differentiating
the Hilbert series ${\rm Ch}_\tau(1,g,t)$ with respect to $g$, hence are rational functions of $g\nu(t)$. Thus, by
Lemma \ref{normfor}, ${\rm Ch}_\tau({\mathbf a},g,t)$ is a rational function for all ${\mathbf a}\in {\mathbf A}$ (giving a direct proof of Theorem \ref{ratfun} in this case). Finally, for generic $g\in {\mathbf T}$ the dimension of $HH_0({\mathbf A},{\mathbf A} g)$ equals $|W|$, as follows, for example, from~\cite{So}. Thus the map ${\rm Ch}|_{t=1}$ in this case is also an isomorphism, and we again obtain all twisted traces from category~$\mathcal{O}$. Finally, it follows from Section~\ref{maxquot} that Weil generically these traces are nondegenerate, hence define nondegenerate short star-products. So we see that nondegenerate short star-products (for fixed $\lambda$ and $g$) depend on $|W|-1$ parameters. Thus, the total number of parameters for such star-products (up to conjugation by the group $G$) is $2\operatorname{rank}(\g)+|W|-1$.

A similar analysis applies to the case when $X$ is a Richardson orbit in~$\g^*$, as in Remark~\ref{rich}. In this case, instead of Verma modules we should use parabolic Verma modules.
\end{Example}

\begin{Remark} Note that in Examples~\ref{exa1},~\ref{exa2} we have $s\in{\mathbf T}$, hence $s$ acts trivially on $HH_0({\mathbf A},{\mathbf A} g)$ and the $s$-invariance condition for twisted traces is vacuous.
\end{Remark}

\begin{Example} Let $X$ admit a symplectic resolution $\widetilde{X}$, and ${\mathbf T}$ act on $\widetilde{X}$ with $N<\infty$ fixed points, hence $X^{{\mathbf T}}=0$ (the situation of~\cite{BLPW}). For simplicity assume that $s\in {\mathbf T}$.
This includes Example~\ref{exa2} ($\widetilde X=T^*G/B$, the Springer resolution),
and also Example~\ref{exa1} for $G=S_n\ltimes \Gamma^n$ where $\Gamma\subset {\rm SL}_2({\mathbb C})$ is a finite subgroup. In this case, it follows from~\cite{BLPW} that for generic $\lambda$ the algebra $C_\nu({\mathbf A})$ is a commutative semisimple algebra of dimension $N$, and for generic $g$ we have $\dim HH_0({\mathbf A},{\mathbf A} g)=N$, so the map ${\rm Ch}|_{t=1}$ is an isomorphism. Thus, traces in representations from category $\mathcal{O}$ give rise to all the $g$-twisted traces.

Also by localizing to the fixed points it can be shown in this case that the Weil generic $g$-twisted trace is nondegenerate. Thus, in this case for generic~$\lambda$, $g$ we obtain a family of nondegenerate short star-products with $N-1$ parameters. This will be discussed in more detail in a subsequent paper.
\end{Example}

\subsection*{Acknowledgements} The first author is grateful to C.~Beem and A.~Kapustin for introducing him to the problem, M.~Kontsevich for contributing a key idea, and C.~Beem, M.~Dedushenko, D.~Gaiotto, D.~Kaledin, D.~Kazhdan, I.~Losev, G.~Lusztig, H.~Nakajima, E.~Rains, L.~Rastelli and T.~Schedler for useful discussions. The work of the first author was partially supported by the NSF grant DMS-1502244.
The work of the second author was supported by the MIT UROP (Undergraduate Research Opportunities Program).

\pdfbookmark[1]{References}{ref}
\LastPageEnding

\end{document}